\documentclass[a4paper,11pt]{amsart}
\usepackage{amsmath}
\usepackage{url,graphicx}
\usepackage{amssymb, color, pstricks-add, manfnt}
\usepackage{amsmath, amsthm, amsfonts}
\usepackage{mathrsfs}
\usepackage{cite}
\usepackage{enumerate}
\usepackage[none]{hyphenat}
\usepackage{multirow}
\usepackage{xcolor}
\usepackage{colortbl}
\usepackage{tabularx}

\theoremstyle{plain}
\newtheorem{Theorem}{Theorem}[section]
\newtheorem{Lemma}{Lemma}[section]

\newtheorem{Corollary}{Corollary}[section]
\newtheorem{Definition}{Definition}[section]
\newtheorem{Example}{Example}[section]

\theoremstyle{remark}
\newtheorem{remark}{Remark}

\numberwithin{equation}{section}
\numberwithin{figure}{section}
\numberwithin{remark}{section}

\usepackage{autobreak}
\allowdisplaybreaks

\usepackage{hyperref}
\usepackage{cite}

\begin{document}
\begin{sloppypar}

	\title{Improved regularity estimates for degenerate or singular fully nonlinear dead-core systems and H\'{e}non-type equations}
	
	\author{Jiangwen Wang}
	\address{School of Mathematics and Shing-Tung Yau Center of Southeast University, Southeast University, Nanjing 211189, P.R. China}
	\email{jiangwen\_wang@seu.edu.cn}	

	\author{Feida Jiang$^*$}
	\address{School of Mathematics and Shing-Tung Yau Center of Southeast University, Southeast University, Nanjing 211189, P.R. China; Shanghai Institute for Mathematics and Interdisciplinary Sciences, Shanghai 200433, P.R. China}
	\email{jiangfeida@seu.edu.cn}	
	

	\date{\today}
	\thanks{*corresponding author}
	

	\begin{abstract}
	In this paper, we study the degenerate or singular fully nonlinear dead-core systems coupled with strong absorption terms. We establish several properties, including improved regularity of viscosity solutions along the free boundary, non-degeneracy, a measure estimate of the free boundary, Liouville-type results, and the behavior of blow-up solution. We also derive sharp regularity estimates for viscosity solutions to H\'{e}non-type equations with a degenerate weight and strong absorption, governed by a degenerate fully nonlinear operator. Our results are new even for the model equations involving degenerate Laplacian operators.
\end{abstract}

	\maketitle		

\section{Introduction}
The first topic of this paper is to study the regularity properties of viscosity solutions to degenerate or singular fully nonlinear dead-core systems with strong absorption terms
\begin{equation}\tag{{\bf DCS}}
\label{DCP}
\left\{
     \begin{aligned}
     & |Du|^{p} F(D^{2}u,x) = (v_{+})^{\lambda_{1}}   \quad \text{in} \ \ B_{1}      \\
     &  |Dv|^{q} G(D^{2}v,x) =  (u_{+})^{\lambda_{2}}   \quad \text{in} \ \ B_{1},    \\
     \end{aligned}
     \right.
\end{equation}
where $ u_{+} = \max \{u,0\} $, $ v_{+} = \max \{v,0\} $, and $ F $ and $ G$ are fully nonlinear uniformly elliptic operator with uniformly continuous coefficients, namely,

\label{Section1:ass1} {\bf (A1) (Uniformly ellipticity of $ F, G$).} For any $  \mathrm{N} \geq 0 $, $ \mathrm{M} \in \mathrm{Sym}(n)$, $ x,y \in B_{1} $, there holds
\begin{equation}\label{1a}
\lambda ||\mathrm{N}|| \leq F(\mathrm{M}+\mathrm{N}, x) - F(\mathrm{M}, x) \leq \Lambda ||\mathrm{N}|| \tag{{\bf 1a}}
\end{equation}
and
\begin{equation*}
\lambda ||\mathrm{N}|| \leq G(\mathrm{M}+\mathrm{N}, x) - G(\mathrm{M}, x) \leq \Lambda ||\mathrm{N}||,
\end{equation*}
with $ 0 < \lambda \leq \Lambda $.

{\bf (A2) (Continuity on the coefficients of $ F, G $).}

--\label{Section1:ass2a} {\bf (A2a).} We shall assume
\begin{equation}\label{2a}
   | F(\mathrm{M}, x) - F(\mathrm{M}, y)  | \leq C\omega_{F}(|x-y|)(1+||\mathrm{M}||)    \tag{{\bf 2a}}
\end{equation}
and
\begin{equation*}
  | G(\mathrm{M}, x) - G(\mathrm{M}, y)  | \leq C\omega_{G}(|x-y|)(1+||\mathrm{M}||)
\end{equation*}
for all $ \mathrm{M} \in  \mathrm{Sym}(n)  $ and $ x,y \in B_{1} $, where $ C >0 $ is a constant and $\omega_{F},  \omega_{G} $ are modulus of continuity.

--\label{Section1:ass2b} {\bf (A2b).} We impose stronger assumptions on the coefficients of $ F, G $ than those in    \hyperref[Section1:ass2a]{\bf (A2a)}. Specifically, we assume
\begin{equation*}
  | F(\mathrm{M}, x) - F(\mathrm{M}, y)  | \leq C\omega_{F}(|x-y|)(1+||\mathrm{M}||)^{\tau}
\end{equation*}
and
\begin{equation*}
  | G(\mathrm{M}, x) - G(\mathrm{M}, y)  | \leq C\omega_{G}(|x-y|)(1+||\mathrm{M}||)^{\tau},
\end{equation*}
for all $ \mathrm{M} \in  \mathrm{Sym}(n)  $ and $ x,y \in B_{1} $, where $ 0 < \tau < \overline{\alpha} $ and $ C >0 $ is a constant. Here $ \omega_{F} $ and $ \omega_{G} $ denote modulus of continuity that are H\"{o}lder continuity. Namely, there exist constants $ K_{1}, K_{2}> 0 $ and $ \overline{\alpha} \in (0,1) $ such that
\begin{equation*}
  \omega_{F}(|x-y|) \leq K_{1}|x-y|^{\overline{\alpha}} \quad \mathrm{and}  \quad  \omega_{G}(|x-y|) \leq K_{2}|x-y|^{\overline{\alpha}}.
\end{equation*}

In addition, we also impose certain necessary conditions on $ p, q, \lambda_{1} $ and $ \lambda_{2} $.

\label{Section1:ass3} {\bf (A3) ($ p,q $ and order of reaction $ \lambda_{1}, \lambda_{2}$).} We suppose that
\begin{equation*}
      -1 <  p, q < \infty, \quad \lambda_{1}, \lambda_{2} \geq 0 \quad  \mathrm{and} \quad \lambda_{1}\lambda_{2} < (1+p)(1+q).
\end{equation*}

In recent years, reaction-diffusion equations characterized by significant absorption terms have attracted considerable attention due to their applications in various fields, including physics, material sciences and chemical engineering. A key aspect of these models is the presence of dead-core regions where the density of a substance (or gas) drops to zero (cf. \cite{A175, A275, BSS84, HM85} for some motivational works). When the density of one substance is affected by the density of another, the phenomenon is described as a dead-core system. An interesting example is given by
\begin{equation*}
\left\{
     \begin{aligned}
     & F(D^{2}u,x) =  f(u, v, x)   \quad      \text{in}    \ \  B_{1}          \\
     &   G(D^{2}v,x) = g(u, v, x)   \quad      \text{in}   \ \  B_{1},            \\
     \end{aligned}
     \right.
\end{equation*}
where the operators $ F, G: \mathrm{Sym}(n) \times B_{1} \rightarrow \mathbb{R} $ and $ f, g :\mathbb{R} \times \mathbb{R} \times B_{1} \rightarrow  \mathbb{R} $ represent the convection terms of the model.

The study of dead-core free boundary problems, such as \eqref{DCP}, has attracted significant interest over the past four decades. Extensive research has been conducted on various aspects of this topic, including the existence of solutions, the characterization of dead-core sets, and the asymptotic behavior of solutions (see, for instance, Pucci--Serrin \cite{PS06} and Bandle {\em et al.}\cite{BV03}). Although there exists a substantial body of literature addressing dead-core problems in the divergence form, the quantitative properties associated with nonlinear systems remain comparatively underexplored. Thereby this has served as the primary motivation for the research in the current paper.

The system \eqref{DCP} has been analyzed in three specific cases, namely,

\label{H1} {\bf (H1)}. $ p=q=0, u=v, F=G \ \text{and} \ \lambda_{1} = \lambda_{2} $;

\label{H2} {\bf (H2)}. $ -1< p=q < \infty, F=G \ \text{and} \ \lambda_{1} = \lambda_{2} $;

\label{H3} {\bf (H3)}. $ p = q =0 $.

The system \eqref{DCP} in the \hyperref[H1]{\bf (H1)} case corresponds to the single fully nonlinear equation with strong absorption term
\begin{equation}\label{Intro:eq2}
F(D^{2}u,x) = (u_{+})^{\lambda_{1}} \ \ \text{in} \ \ B_{1}.
\end{equation}
The Krylov--Safonov regularity theory in \cite{CC95} can deduce that viscosity solutions of \eqref{Intro:eq2} are locally of the class $ C^{1,\alpha} $, for some $ \alpha \in (0,1) $. Teixeira in \cite{T16}, using an improved flatness lemma, scaling techniques and iterative argument, derived the optimal $ C^{\frac{2}{1-\lambda_{1}}}$ regularity along the free boundary $ \partial \{ u >0\} $. Later, Ara\'{u}jo {\em et al.} \cite{ALT16} employed the same methods as \cite{T16} to establish the sharp $ C^{\frac{4}{3-\lambda_{1}}}$ regularity for the dead-core problem ruled by the infinity Laplacian, $
\Delta_{\infty}u(x) = (u_{+})^{\lambda_{1}} \ \ \text{in} \ \ B_{1} $, where $ \lambda_{1} \in [0, 3) $ is a constant.

For the \hyperref[H2]{\bf (H2)} case we noticed that \eqref{DCP} is related to the degenerate or singular fully nonlinear equation
\begin{equation}\label{Intro:eq3}
   |Du|^{p}F(D^{2}u,x) = (u_{+})^{\lambda_{1}} \ \ \text{in} \ \ B_{1},
\end{equation}
where $ 0 \leq \lambda_{1} < 1+p $. Based on a novel flatness improvement in \cite{SLR21}, da Silva {\em et al.} also proved a higher $ C^{\frac{2+p}{1+p-\lambda_{1}}}$ regularity for dead-core solution to \eqref{Intro:eq3} across the free boundary, and established the Liouville type theorem. Clearly, the result obtained in \cite{T16} is a special case of \cite{SLR21}. In the case when $ p \geq 0, -1 < \lambda_{1} < 1+p $ in \eqref{Intro:eq3}, Teixeira \cite{T18} obtained quantitative regularity estimates for non-negative limiting solutions to \eqref{Intro:eq3}. In the parabolic setting, we refer the readers to da Silva {\em et al.}\cite{SS18, SOS18, SO19}.

For the \hyperref[H3]{\bf (H3)} case the system \eqref{DCP} reduces to a class of fully nonlinear elliptic system without degeneracy or singularity. This model was first studied by Ara\'{u}jo {\em et al.}\cite{AT24}. Unlike the single equations \eqref{Intro:eq2}--\eqref{Intro:eq3}, when solving the system, one may have several difficulties, such as, lacking the comparison principle and being unable to use Perron's method in the framework of \cite{T16, SLR21}. Nevertheless, the authors in \cite{AT24} overcame these obstacles using Schaefer's fixed point theorem, and obtained some weak geometric results such as non-degeneracy, Liouville type results and the measure estimate of free boundary $  \partial \{|(u,v)| >0\}  $, where $ |(u,v)|:= (u_{+})^{\frac{1}{1+\lambda_{1}}} + (v_{+})^{\frac{1}{1+\lambda_{2}}}  $.

\vspace{1mm}

Motivated by these works \cite{T16, ALT16, T18, SLR21, AT24}, the primary objective of the first part of this work is to investigate geometric and analytic properties of nonnegative solutions of \eqref{DCP}. To the best of our knowledge, there are very few results for properties of viscosity solutions to \eqref{DCP}, even for the degenerate or singular Laplacian systems, i.e., $ F= G = \Delta $ in \eqref{DCP}. Moreover, our work unifies and extends the previous results from \hyperref[H1]{\bf (H1)}--\hyperref[H3]{\bf (H3)} in a broader setting.

We believe that \hyperref[H1]{\bf (H1)}--\hyperref[H3]{\bf (H3)}, along with its generalization in \eqref{DCP}, has a fundamental connection to various free boundary problems. These connections span both variational (cf. Alt-Phillips \cite{AP86}, Friedman-Phillips \cite{FP84} and De Silva-Savin \cite{1DeS23, 2DeS23, 3DeS23}) and non-variational (cf. Wu-Yu \cite{WY22}) settings. Therefore the study of \eqref{DCP} is of significant mathematical interest.

The first main result of our article concerns sharp and improved regularity of dead-core solutions to \eqref{DCP} along the free boundary $  \partial \{|(u,v)| >0\}  $, where
\begin{equation*}
  |(u,v)|:= (u_{+})^{\frac{2}{(1+q)(2+p)+\lambda_{1}(2+q)}} + (v_{+})^{\frac{2}{(1+p)(2+q)+\lambda_{2}(2+p)}}.
\end{equation*}

For simplicity, hereafter we say $ (u,v) \geq 0 $ if $ u, v \geq 0 $.

\begin{Theorem}[{\bf Improved regularity along free boundary}]\label{Thm1}
Suppose that \hyperref[Section1:ass1]{\bf (A1)}, \hyperref[Section1:ass2a]{\bf (A2a)} and \hyperref[Section1:ass1]{\bf (A3)} hold. Let $ (u,v) \geq 0 $ be a bounded viscosity solution of \eqref{DCP} and $ x_{0} \in \partial \{ |(u,v)|>0 \} \cap B_{1/2} $, then there exists a constant $ C = C(p, q, \lambda_{1}, \lambda_{2}, ||u||_{L^{\infty}(B_{1})}, ||v||_{L^{\infty}(B_{1})}) >0 $ such that
\begin{equation*}
  |(u,v)| \leq C |x-x_{0}|^{\frac{2}{(1+p)(1+q)-\lambda_{1}\lambda_{2}}}
\end{equation*}
for any $ x \in B_{1/8} $. Particularly, we have
\begin{equation*}
  |u(x)| \leq C|x-x_{0}|^{\frac{(1+q)(2+p)+\lambda_{1}(2+q)}{(1+p)(1+q)-\lambda_{1}\lambda_{2}}} \ \ \text{and} \ \ |v(x)| \leq C|x-x_{0}|^{\frac{(1+p)(2+q)+\lambda_{2}(2+p)}{(1+p)(1+q)-\lambda_{1}\lambda_{2}}}.
\end{equation*}
\end{Theorem}

We will also discuss how dead-core solutions behave away from their free boundary $  \partial \{|(u,v)| >0\}  $. Before presenting the second result of our paper, we list an extra assumption:

\label{Section1:ass4} {\bf (A4).}
\begin{equation}
\label{3a}
   F(\mathrm{O}_n,x) =  0,  \quad    \forall \ x
 \in B_{1},     \tag{{\bf 3a}}
\end{equation}
and
\begin{equation*}
  G(\mathrm{O}_n,x) = 0,  \quad    \forall \ x
 \in B_{1}.
\end{equation*}

\begin{Theorem}[{\bf Non-degeneracy}]
\label{Thm2}
Suppose that \hyperref[Section1:ass1]{\bf (A1)}, \hyperref[Section1:ass2b]{\bf (A2b)}, \hyperref[Section1:ass3]{\bf (A3)} and \hyperref[Section1:ass4]{\bf (A4)} hold. Let $ (u,v) \geq 0 $ be a bounded viscosity solution to \eqref{DCP}, and $ x_{0} \in \overline{\{|(u,v)|>0\}} \cap B_{1/2}  $, then for a constant $ c > 0 $, depending on $ n, \Lambda, p, q, \lambda_{1} $ and $ \lambda_{2}$, such that
 \begin{equation*}
   \sup_{\overline{B_{r}(x_{0})}} |(u,v)| \geq c r^{\frac{2}{(1+p)(1+q)-\lambda_{1}\lambda_{2}}}
 \end{equation*}
for any $ r \in (0,\frac{1}{2}) $.
\end{Theorem}

Before proceeding further, we conclude the following observations in Remarks \ref{Intro:rk3}--\ref{Intro:rk2}.

\begin{remark}\label{Intro:rk3}
Thanks to Theorems \ref{Thm1} and \ref{Thm2}, one has
\begin{equation*}
  u \in C^{\frac{(1+q)(2+p)+\lambda_{1}(2+q)}{(1+p)(1+q)-\lambda_{1}\lambda_{2}}} \ \ \text{and} \ \ v \in C^{\frac{(1+p)(2+q)+\lambda_{2}(2+p)}{(1+p)(1+q)-\lambda_{1}\lambda_{2}}}
\end{equation*}
along the free boundary $  \partial \{|(u,v)| >0\}  $. This regularity is higher than the regularity of viscosity solutions to degenerate or singular equation. As a matter of fact, for the case $ p,q \geq 0 $, as $ |Du|^{p} F(D^{2}u,x) \in C^{0,\lambda_{1}}_{\mathrm{loc}} $ and $ |Dv|^{q} G(D^{2}v,x) \in C^{0,\lambda_{2}}_{\mathrm{loc}} $. Then Lemma \ref{Section2:lem2} assures that
\begin{equation*}
  u \in C^{1,\min\{\alpha_{0}^{-},\frac{1+\lambda_{1}}{1+p} \}}_{\mathrm{loc}} \ \ \text{and} \ \ v \in C^{1,\min\{\alpha_{0}^{-},\frac{1+\lambda_{2}}{1+q} \}}_{\mathrm{loc}}.
\end{equation*}
When the operator $ F, G $ are concave/convex and $ \lambda_{1} < p $, it was proved in \cite[Theorem 1]{T24} sharp regularity, i.e.,
\begin{equation*}
  u \in C^{1,\frac{1+\lambda_{1}}{1+p}}_{\mathrm{loc}}   \ \ \text{and} \ \ v \in C^{1,\frac{1+\lambda_{2}}{1+q}}_{\mathrm{loc}}.
\end{equation*}
Nevertheless, we must highlight that
\begin{equation*}
  \frac{(1+q)(2+p)+\lambda_{1}(2+q)}{(1+p)(1+q)-\lambda_{1}\lambda_{2}} > 1+  \frac{1+\lambda_{1}}{1+p} \ \ \text{and}   \ \ \frac{(1+p)(2+q)+\lambda_{2}(2+p)}{(1+p)(1+q)-\lambda_{1}\lambda_{2}} > 1+ \frac{1+\lambda_{2}}{1+q},
\end{equation*}
which implies that dead-core solution $ (u,v) $ is more regular, across the free boundary, than the best regularity result coming from the classical regularity theory. Analogously, for the case $ -1 < p, q < 0 $, it can be seen that $ u, v \in C^{1,\alpha}_{\mathrm{loc}} $ for any $ 0 < \alpha < 1 $ in the spirit of Remark \ref{Sec2:rk2}. Direct calculations also yield that
 \begin{equation*}
  \frac{(1+q)(2+p)+\lambda_{1}(2+q)}{(1+p)(1+q)-\lambda_{1}\lambda_{2}}, \ \frac{(1+p)(2+q)+\lambda_{2}(2+p)}{(1+p)(1+q)-\lambda_{1}\lambda_{2}} > 1+  \alpha.
\end{equation*}
\end{remark}

\begin{remark}\label{Intro:rk1}
As a direct consequence of Theorem \ref{Thm1} and \ref{Thm2}, one has $ |(u,v)| $ grows exactly with rate $ |x|^{\frac{2}{(1+p)(1+q)-\lambda_{1}\lambda_{2}}} $, see Corollary \ref{Se6:coro1}, for details.
\end{remark}

\begin{remark}\label{Intro:rk2}
The method here may also work for a much broader class of degenerate or singular fully nonlinear systems
\begin{equation*}
\left\{
     \begin{aligned}
     & F(D^{2}u,Du,x) + |Du|^{p}\langle b_{1}(x), Du\rangle= (v_{+})^{\lambda_{1}} \ \ \text{in} \ \ B_{1}   \\
     & G(D^{2}v,Dv,x) + |Dv|^{q} \langle b_{2}(x), Dv\rangle=  (u_{+})^{\lambda_{2}} \ \ \text{in} \ \ B_{1},        \\
     \end{aligned}
     \right.
\end{equation*}
where $ b_{1}, b_{2} \in C^{0}(\overline{B}_{1}, \mathbb{R}^{n}) $, and the operator $ F, G $ satisfy some particular conditions in \cite{SLR21}. A full exploration of this issue is reserved for a subsequent paper.
\end{remark}

As an application of our findings, we prove some Liouville-type results for problem \eqref{DCP}. More precisely,

\begin{Theorem}[{\bf Improved Liouville-type result}]
\label{Section7:Thm2}
Suppose that \hyperref[Section1:ass1]{\bf (A1)}, \hyperref[Section1:ass2b]{\bf (A2b)}, \hyperref[Section1:ass3]{\bf (A3)} and \hyperref[Section1:ass4]{\bf (A4)} hold. Let $ (u,v) $ be a non-negative viscosity solution of
\begin{equation}\label{Section7:eqq1}
\left\{
     \begin{aligned}
     &  |Du|^{p} F(D^{2}u,x) = (v_{+})^{\lambda_{1}}\ \ \text{in} \ \ \mathbb{R}^{n}    \\
     &  |Dv|^{q} G(D^{2}v,x) =  (u_{+})^{\lambda_{2}} \ \ \text{in} \ \ \mathbb{R}^{n},       \\
     \end{aligned}
     \right.
\end{equation}
and
\begin{equation}\label{Section7:eq5}
\limsup_{|x|\rightarrow \infty} |x|^{\frac{-2}{(1+p)(1+q)-\lambda_{1}\lambda_{2}}} |(u(x), v(x))| <  \min \bigg\{A^{\frac{2}{(1+q)(2+p)+\lambda_{1}(2+q)}}, B^{\frac{2}{(1+p)(2+q)+\lambda_{2}(2+p)}} \bigg\},
\end{equation}
where
\begin{equation*}
\left\{
     \begin{aligned}
     & \alpha := \frac{(1+q)(2+p)+\lambda_{1}(2+q)}{(1+p)(1+q)-\lambda_{1}\lambda_{2}}; \ \beta :=  \frac{(1+p)(2+q)+\lambda_{2}(2+p)}{(1+p)(1+q)-\lambda_{1}\lambda_{2}} \\
     &  A = \big[ \Lambda(n+\beta-2)\beta^{1+q} \big]^{\frac{\lambda_{1}}{\lambda_{1}\lambda_{2}-(1+p)(1+q)}}  \big[ \Lambda(n+\alpha-2)\alpha^{1+p} \big]^{\frac{1+q}{\lambda_{1}\lambda_{2}-(1+p)(1+q)}}   \\
     &  B = \big[ \Lambda(n+\alpha-2) \alpha^{1+p} \big]^{\frac{\lambda_{2}}{\lambda_{1}\lambda_{2}-(1+p)(1+q)}} \big[ \Lambda(n+\beta-2) \beta^{1+q} \big]^{\frac{1+p}{\lambda_{1}\lambda_{2}-(1+p)(1+q)}}.       \\
     \end{aligned}
     \right.
\end{equation*}
Then $ u \equiv v \equiv 0 $ in $ \mathbb{R}^{n} $.
\end{Theorem}

\begin{remark}
Note that Theorem \ref{Section7:Thm2} is only applicable to $ 0\leq  p,q < \infty $. We will further explain this in Section \ref{Section 6}.
\end{remark}

The second topic of this paper is to study the regularity properties of viscosity solutions of H\'{e}non-type equations. Very recently, Hardy-H\'{e}non-type equations with strong absorption term have invoked much attention. Historically, the H\'{e}non equation originated from the study of the stability of a spherical star system in equilibrium under spherical disturbances (see, for instance, \cite{H73}). To recall a classical elliptic model, consider:
 \begin{equation}\label{***}
\left\{
     \begin{aligned}
     & -\Delta u(x) = |x|^{\alpha} u^{p}(x)                 \ \ \text{in}    \ \ B_{1}, \ \ p>1,          \\
     &  u(x) >0      \qquad \qquad \qquad \ \text{in} \ \ B_{1},        \\
     &  u(x) = 0     \qquad \qquad \qquad \  \text{on} \ \ \partial B_{1}.
     \end{aligned}
     \right.
\end{equation}
This equation is referred to as a Hardy-type equation when $ \alpha < 0 $, due to its connection with the Hardy-Sobolev inequality, and as a H\'{e}non-type equation if $ \alpha >0 $, owing to its relationship with various diffusion phenomena. While the existence theory for nontrivial solutions to \eqref{***} is well-established and relatively mature, regularity estimates for viscosity solutions to Hardy-H\'{e}non-type equations with possibly singular weights and strong absorption remain largely unexplored.

The first breakthrough in this direction was achieved in the recent work of da Silva {\em et al.} \cite{SPRS24}, where they established sharp and improved regularity estimates for weak solutions of weighted quasilinear elliptic models of Hardy-H\'{e}non-type equations. Subsequently, Bezerra J\'{u}nior {\em et al.} \cite{JSNS24} obtained similar results for equations governed by the infinity Laplacian.

The preceding research (cf. \cite{SPRS24, JSNS24}) raises the important question: can we obtain improved regularity estimates for viscosity solutions to H\'{e}non-type equations with possibly degenerate weights and strong absorption, governed by degenerate fully nonlinear operators? In this paper, we provide a clear affirmative answer to this question. To begin with, let us provide a clear statement. We consider
\begin{equation}\label{Intro:eq4}
|Du|^{p} F(D^{2}u,x) = f(|x|,u(x)) \ \ \text{in} \ \ B_{1}, \tag{{\bf HHTE}}
\end{equation}
where $ f $ fulfills the following assumption:

\label{Section1:ass5} {\bf (A5).} For all $ (x,t) \in B_{1} \times I $, $ r,s \in (0,1) $, where $ I \subset \mathbb{R}$ is an interval, and assume that there exist a positive constant $ C_{n} $ and $ f_{0} \in L^{\infty}(B_{1}) $ such that
$$  |f(r|x|, st)| \leq C_{n} r^{\alpha} s^{\mu} |f_{0}(x)|  \ \  \text{for}  \ \  0 \leq \mu < 1+p, \ 0 \leq p < \infty \ \ \text{and} \ \ 0 < \alpha < \infty. $$

As we will discuss later in this section, the result is novel even for the simplest model $ |Du|^{p} \Delta u = |x|^{\alpha}(u_{+})^{\mu}  $. The fourth main result of this paper is then stated as follows.

\begin{Theorem}[{\bf Higher regularity estimate}]
\label{Thm5}
Let $ u \in C^{0}(B_{1}) $ be a bounded viscosity solution of \eqref{Intro:eq4}, and \eqref{1a}, \eqref{2a} and \hyperref[Section1:ass5]{\bf (A5)} hold.
Then for any point $ x_{0} \in B_{1/2} \cap \partial \{u>0\}  $, such that $ f(|x|,t) \simeq   |x-x_{0}|^{\alpha}t_{+}^{\mu} $, we have
\begin{equation}\label{Intro:eq6}
    \sup_{x\in B_{r}(x_{0})} u(x)  \leq C r^{\frac{2+p+\alpha}{1+p-\mu}}
\end{equation}
for $ r \in (0,1/2) $, where $ C $ is a positive constant depending only on $ n, p, \alpha $ and $ \mu $.
\end{Theorem}

\begin{remark}
If the operator $ F $ is convex/concave, by applying Lemma \ref{Section2:lem1} from Section \ref{Section 2}, we obtain that the bounded viscosity solutions $ u $ of \eqref{Intro:eq4} are locally of class $ C^{1,\frac{1}{1+p}} $. In this setting, we still achieve improved regularity estimate along the free boundary, since
\begin{equation*}
 0 < \alpha < \infty  \ \ \Rightarrow  \ \  \frac{2+p+\alpha}{1+p-\mu} > 1+ \frac{1}{1+p} , \ \ 0 \leq p < \infty.
\end{equation*}
\end{remark}

An immediate and fundamental consequence of Theorem \ref{Thm5} is the local growth estimates for first derivative estimate of solutions. For this purpose, for any point $ z \in \{u>0\} \cap B_{1/2}  $, we define $ \gamma(z) \in \partial \{u>0\} $ such that
\begin{equation*}
  |z-\gamma(z)|= \mathrm{dist}(z,\partial\{u>0\}).
\end{equation*}

\vspace{-5pt}

\begin{Corollary}[{\bf Gradient growth near free boundary points}]
\label{Coro1}
Let $ u \in C^{0}(B_{1}) $ be a bounded viscosity solution to \eqref{Intro:eq4}, and \eqref{1a}, \eqref{2a} and \hyperref[Section1:ass5]{\bf (A5)} hold. Then $ u $ is locally Lipschitz continuous. Furthermore, for any point $ x_{0} \in \{u>0\} \cap B_{1/2}  $ such that $ f(|x|,t) \simeq |x-\gamma(x_{0})|^{\alpha}t_{+}^{\mu}  $, it holds
\begin{equation*}
  |Du(x_{0})| \leq C \mathrm{dist}(x_{0}, \partial \{u>0\})^{\frac{1+\alpha+\mu}{1+p-\mu}},
\end{equation*}
where $ C $ is a positive constant depending only on $ n, p, \alpha $ and $ \mu $.
\end{Corollary}

Hereafter, we denote the critical set of solutions
\begin{equation*}
  \mathcal{S}_{u}(B_{1}):= \{ x \in B_{1}: u(x) = |Du(x)| = 0 \}.
\end{equation*}

To analyze non-degeneracy at critical points, we restrict our attention to the class of nonnegative limiting solutions $ u $, i.e., those obtained as uniform limits of solutions to the penalized problem
\begin{equation*}
\left\{
     \begin{aligned}
     & |Du_{\epsilon}|^{p} F(D^{2}u_{\epsilon},x) = f(|x|, u_{\epsilon}(x)) +  \epsilon        \quad \ \mathrm{in} \ \   B_{1}            \\
     &  u_{\epsilon} =  u   \qquad \qquad \qquad \qquad \qquad \qquad  \  \  \  \quad  \   \mathrm{on} \ \ \partial B_{1}.                     \\
     \end{aligned}
     \right.
\end{equation*}

Here we recall the definition of {\it limiting solution}, which was originally proposed in \cite{JSNS24}.  
\begin{Definition}[Limiting solution]
 Let $\mathcal{O}\subset \mathbb{R}^n$ be an open set. A nonnegative function $ u \in C^{0}(\mathcal{O})$ is said to be a limiting solution of the non-variational PDE
\begin{equation*}
|Du|^{p} F(D^{2}u,x) = f(|x|, u(x))
\quad   \text{in} \quad  \mathcal{O},
\end{equation*}
if there exists a sequence of nonnegative functions $\{u_j\}_{j\in\mathbb{N}}$ satisfying
\begin{equation*}
|Du_{j}|^{p} F(D^{2}u_{j},x) = f(|x|, u_{j}(x)) + (1/j)
\quad \text{in} \quad  \mathcal{O},
\end{equation*}
in the viscosity sense, such that $u_j$ converges locally uniformly to $u$ in $\mathcal{O}$ as $j\to\infty$.
\end{Definition}

We now establish the fifth main result, concerning non-degeneracy at critical points, which reads as follows

\begin{Theorem}[{\bf Non-degeneracy at critical points}]
\label{Thm6}
Let $ u \in C^{0}(B_{1}) $ be a bounded viscosity limiting solution to \eqref{Intro:eq4}, and \eqref{1a}, \eqref{2a} and \hyperref[Section1:ass5]{\bf (A5)} hold. Then there exists $ r^{*} >0 $ such that for any critical point $ x_{0} \in \mathcal{S}_{u}(B_{1}) $ and for all $ r \in (0, r^{*}) $ such that $ B_{r}(x_{0}) \subset B_{1} $, it holds
\begin{equation*}
  \sup_{\partial B_{r}(x_{0})} u(x)  \geq \bigg[\frac{(1+p-\mu)^{2+p}}{\Lambda(2+p+\alpha)^{1+p}[n(1+p-\mu)+(2\mu+\alpha-p)]}            \bigg]^{\frac{1}{1+p-\mu}}  r^{\frac{2+p+\alpha}{1+p-\mu}}.
\end{equation*}
\end{Theorem}

Next we will turn our attention to the critical case obtained as $ \mu \rightarrow 1+p $, i.e.,
\begin{equation}\label{Intro: eq7}
|Du|^{p} F(D^{2}u, x) = |x|^{\alpha} u^{1+p} \ \ \text{in} \ \ B_{1}, \ \ 0 \leq p < \infty.
\end{equation}

Utilizing the barrier argument for the critical equation \eqref{Intro: eq7}, Lemma \ref{Section2:lem1} and Remark \ref{Sec2:rk2}, we shall demonstrate that a viscosity solution to \eqref{Intro: eq7} cannot vanish at an interior point, unless it is identically zero.

\begin{Theorem}[{\bf Strong maximum principle}]
\label{Thm7}
Let $ \alpha \in (0, \infty) $ and $ u \in C^{0}(B_{1})$ be a viscosity solution to \eqref{Intro: eq7}. Suppose that there exists an interior point $ x_{0} \in B_{1} $ such that $ u(x_{0}) = 0 $. Then, $ u \equiv 0 $ in $ B_{1} $.
\end{Theorem}

{\bf Difficulties, strategy and novelty of this paper}. Here, we will briefly outline the novel aspects of this paper.

\begin{enumerate}[(1)]

\item It is noteworthy that the approach leading to our findings in \eqref{DCP} is new in the current literature. In contrast to studies of dead-core problems for single equations \cite{T16, SLR21}, one may encounter several challenges. One such difficulty is the absence of a comparison principle, which restricts the construction of suitable barrier functions to examine the non-degeneracy of the solution. However, in this article we provide a version of weak comparison principle, involving with general strong absorption terms, see Lemma \ref{Sec3:le1}. This conclusion generalizes the results in \cite[Lemma 2.1]{AT24}, and leads to non-degeneracy of solution, geometry properties of free boundary and some Liouville type results, as discussed in Sections \ref{Section 5}--\ref{Section 6}.

    \vspace{1mm}

\item Unlike the dead-core problem for fully nonlinear systems \cite{AT24}, our model \eqref{DCP} becomes significantly degenerate or singular along the free boundary $  \partial \{|(u,v)| >0\}  $ when $ |\nabla u| $ and $ |\nabla v| $ vanish at free boundary points. For this reason, we fully leverage the $ C^{1,\alpha'}_{\mathrm{loc}} (0<\alpha'<1) $ regularity of solutions to degenerate or singular fully nonlinear equation, and develop a refined analysis to solve this non-trivial task, see Section \ref{Section 3}. The findings of our study are applicable to systems comprising more than two equations.

    \vspace{1mm}

\item We emphasize that our results for problem \eqref{Intro:eq4} recover previous work (cf. Teixeira \cite{T16}, $ \alpha=0, p=0 $) and complement the result for degenerate fully nonlinear equation (cf. da Silva \cite{SLR21}, $ 0 \leq  p < \infty, 0 \leq \mu < 1+p, \alpha =0 $). To some extent, we obtain sharp and improved regularity estimate along the free boundary via alternative strategies and techniques. In effect, we used the Harnack inequality, Theorem \ref{Pre:Thm2.1} to show Theorem \ref{Thm5}, rather than employing the flatness estimate in Theorem \ref{Thm1}, see Section \ref{Section 7}. Apart from this, compared with \cite{T16} and \cite{SLR21}, Theorem \ref{Thm5} reveals an insightful and deep significance: $ \alpha $ and $ \mu $ can jointly    influence the regularity estimate of viscosity solution to \eqref{Intro:eq4} around the set of critical points, see Example \ref{example1} and Remark \ref{rk91} below.

    \vspace{1mm}

\item Finally, our results are novel even for the simpler prototype given by the operator $ F = G = \Delta $ in \eqref{DCP}, and $ p=0, f(|x|, u) = |x|^{\alpha}(u)_{+}^{\mu} $ in \eqref{Intro:eq4}. To summarise, the following Table \ref{table1} can explain how our result of this paper generalize the results of the previous works \cite{T16, SLR21, AT24}.

    \vspace{-5pt}

\begin{table}[!htp]
\newcommand{\tabincell}[2]{\begin{tabular}{@{}#1@{}}#2\end{tabular}}
   \centering
   {
    \linespread{0.27}  \selectfont
    \resizebox{\textwidth}{!}{
   \begin{tabular}{||c|c|c|c||}
        \hline
        Model PDEs & Compatibility condition & \tabincell{c}{Improved and sharp  \\ regularity estimates} & References   \\
        \hline
         $ F(D^{2}u,x) =  u_{+}^{\lambda_{1}} $   &  $ 0 \leq \lambda_{1} < 1 $    &  $  u \in C^{\frac{2}{1-\lambda_{1}}}_{\text{loc}} $  & \cite{T16}       \\
        \hline
          $ |Du|^{p} F(D^{2}u,x) = u_{+}^{\lambda_{1}} $  &  $  0 \leq p < \infty \ \text{and} \ 0 \leq \lambda_{1} < 1+p $      &      $ u \in C^{\frac{2+p}{1+p-\lambda_{1}}}_{\text{loc}}    $  &  \cite{SLR21}              \\
         \hline
         $ |Du|^{p} F(D^{2}u,x) = |x|^{\alpha} u_{+}^{\mu} $   &  \tabincell{c}{$ 0 \leq  p < \infty, 0 \leq \mu < 1+p \ \text{and} $ \\ \vspace{1mm}  $  0 \leq  \alpha < \infty $         } & $ u \in C^{\frac{2+p+\alpha}{1+p-\mu}}_{\text{loc}} $ & \text{Theorem} \ref{Thm5}      \\
        \hline
    $  (|Du|^{p} + a(x) |Du|^{q}) F(D^{2}u,x) = u_{+}^{\mu} $  &  \tabincell{c}{$ 0 < p \leq q < \infty, 0 < \mu < 1+p \ \text{and} \ $  \\ \vspace{1mm}  $ 0 \leq a(x) \in C^{0}(B_{1}) $  }    &    $ u \in C_{\text{loc}}^{\frac{2+p}{1+p-\mu}} $  &  \cite{SR20}      \\
     \hline
   $ \Delta_{\infty} u = |x|^{\alpha} u_{+}^{\mu} $  &  $ 0 < \alpha < \infty \ \text{and} \ 0 \leq \mu < 3 $ &   $ u \in C_{\text{loc}}^{\frac{4+\alpha}{3-m}} $  &   \cite{JSNS24}           \\
    \hline
 $ \text{div}(|Du|^{p-2}\mathscr{A}(x)Du) = |x|^{\alpha}u_{+}^{\mu}  $  &  \tabincell{c}{$ \alpha > -1-\mu, p>1 \ \text{and} $   \\ \vspace{1mm}   $   0 \leq \mu < p-1 $}  &   $  u  \in  C_{\text{loc}}^{\frac{p+\alpha}{p-1-\mu}}   $ &  \cite{SPRS24}      \\
    \hline
     $
  \begin{cases}
  F(D^{2}u,x) = (v_{+})^{\lambda_{1}}      \\
  G(D^{2}v,x) =  (u_{+})^{\lambda_{2}}         \\
  \end{cases} $  &   $ \lambda_{1}, \lambda_{2} \geq 0 \ \text{and} \ \lambda_{1} \lambda_{2} <1      $          &
  \tabincell{c}{$ u  \in C^{\frac{2(1+\lambda_{1})}{1-\lambda_{1}\lambda_{2}}}_{\text{loc}} $   \\  \vspace{1mm}
    $ v \in  C^{\frac{2(1+\lambda_{2})}{1-\lambda_{1}\lambda_{2}}}_{\text{loc}} $ }
       &      \cite{AT24}           \\
         \hline
          $
     \begin{cases}
    |Du|^{p} F(D^{2}u,x) = (v_{+})^{\lambda_{1}}      \\
    |Dv|^{q} G(D^{2}v,x) =  (u_{+})^{\lambda_{2}}         \\
     \end{cases}
     $  &  \tabincell{c}{ $ -1< p,q < \infty, \lambda_{1}, \lambda_{2} \geq 0 \  \text{and} \ $
     \\ $ \lambda_{1}\lambda_{2} < (1+p)(1+q) $}  &
  \tabincell{c}{$ u \in C^{\frac{(1+q)(2+p)+\lambda_{1}(2+q)}{(1+p)(1+q)-\lambda_{1}\lambda_{2}}}_{\text{loc}} $  \\  \vspace{1.5mm}
    $ v \in C^{\frac{(1+p)(2+q)+\lambda_{2}(2+p)}{(1+p)(1+q)-\lambda_{1}\lambda_{2}}}_{\text{loc}} $}
         &   \text{Theorem}  \ref{Thm1}           \\
         \hline
   \end{tabular}
   }
\caption{Improved regularity estimates across the free boundary for different PDEs/Systems. Here $ 0< \lambda \textbf{I}\rm{d}_{n}  \leq \mathscr{A}(\cdot) \leq \Lambda \textbf{I}\rm{d}_{n} $ is a symmetric matrix with sufficiently smooth entries (cf. \cite{PS07}).}
\label{table1}
  }
\end{table}
\end{enumerate}

\vspace{-15pt}

{\bf{State-of-the-art.}} In recent years, there has been growing research interest in degenerate or singular operator-driven equations.

Regarding the eigenvalue problem, Birindelli-Demengel in \cite{BD06} studied degenerate or singular equations of the form
\begin{equation*}
  |Du|^{\gamma} F(D^{2}u) = f(x,u) \lesssim \lambda |u|^{\gamma}u,
\end{equation*}
where $ F $ is fully nonlinear uniformly elliptic, $ \gamma > -1 $, and $ \lambda \in \mathbb{R}$. Lipschitz regularity results were also established in \cite{BD06} for bounded $ f $. Subsequently, for the case where $ \gamma > 0 $ and $ f $ is independent of $ u $, Imbert and Silvestre \cite{LS13} derived the interior $ C^{1,\alpha} $ regularity result. These findings have been extended to scenarios where $ f $ depends on the gradient with growth conditions of the form $ 1+ \gamma $, as well as to boundary $ C^{1,\alpha} $ regularity results under the assumption of sufficiently regular boundary data, in the works of \cite{BD15, BD14, BDL19}.

Parallel progress includes the contributions of Baasandorj {\em et al.} \cite{BBLL24, BBO23}, who established optimal $ C^{1,\alpha} $-regularity for viscosity solutions governed by a general class of degenerate or singular operators, along with $ C^{1}$-regularity for certain degenerate or singular fully nonlinear elliptic equations under minimal structural assumptions. Notably, Bezerra J\'{u}nior {\em et al.} \cite{BdaRR23} recently proved global gradient estimates for a class of fully nonlinear PDEs with spatially variable degeneracy. For a comprehensive survey of recent advances in fully nonlinear PDEs with unbalanced degeneracy laws, we refer readers to \cite{BdaRRV22}.

The field has also seen highly influential progress in dead-core problems. Building on preliminary results \cite{T16, ALT16, T18, SLR21, AT24}, two particularly interesting and important results have emerged in this line of research. da Silva {\em et al.} \cite{daRS19} provides regularity estimates for $ p$-dead-core problems and analyzes their asymptotic
behavior as $ p \rightarrow \infty $. More recently, Alcantara {\em et al.} \cite{AdaS25} presents geometric regularity estimates for quasilinear elliptic models in non-divergence form with strong absorption and related H\'{e}non-type equations.

{\bf Organization of the paper}. We organise the remaining content of this paper as follows.

In Section \ref{Section 2}, we recall the definition of viscosity solution to \eqref{DCP} and \eqref{Intro:eq4}, and provide some auxiliary results, which will be frequently used in this article. Particularly, we shall establish a new weak comparison principle for degenerate or singular fully nonlinear elliptic systems in a general setting.

In Section \ref{Section 3}, using a completely new flatness estimate, scaling techniques and an iterative argument, we derive a higher regularity of solutions to \eqref{DCP} at the free boundary points. Subsequently, we provide two consequences of Theorem \ref{Thm1}, which may be of own interest, see Corollary \ref{Section4:coro1} and \ref{Section4:coro2}.

In Sections \ref{Section 4} and \ref{Section 5}, we are devoted to analysis the radial solutions for \eqref{DCP}, and further using the radial solutions we construct and weak comparison principle, see Lemma \ref{Sec3:le1}, to obtain non-degeneracy of solutions. As a result, density estimate and finer geometric measure property are also discussed.

In Section \ref{Section 6}, we shall provide two Lioville-type results for entire solutions and study the limiting profile of blow-up solutions.

In Sections \ref{Section 7} and \ref{Section 8}, we prove the higher regularity along free boundary for equation \eqref{Intro:eq4}, non-degeneracy at critical points and strong maximum principle. Additionally, we also provide some applications, see Corollary \ref{Coro1}, and Theorems \ref{Section8:thm1}--\ref{Section8:thm2}.

In Section \ref{Section9}, we deliver several interesting examples to describe new features. Finally, \hyperref[Appendix A]{Appendix A} contains the proof of Theorem \ref{Se2:Thm2.2}.

{\bf Notations}. The following notations are also used in this article.
\begin{itemize}

\item For $ r >0 $, $ B_{r}(x) $ denotes the open ball of radius $ r $ centered at $ x $. We simply use $ B_{r} $ to denote the open ball $ B_{r}(0)$.

\item  $ \textbf{I}\mathrm{d_{n}} $ denotes the $n\times n$ identity matrix.

\item $ \mathscr{H}^{n}(E) $ denotes the $ n $-dimensional Hausdorff measure of a measurable set $ E $.

\item $ \text{int}(E) $ denotes the set of interior point of $ E $.

\item $ \text{Sym}(n) $ denotes the space of all $ n \times n $ symmetric matrices in $ \mathbb{R}^{n} $.

\item $ f \simeq g $ denotes there exists a constant $ C > 0 $ such that $ \frac{1}{C}g \leq f \leq Cg $.

\item $ C $ shall denote a generic positive constant which may vary in different inequalities.

\item ``a.e." denotes ``almost everywhere".
\end{itemize}

\vspace{3mm}

\section{Preliminaries}\label{Section 2}
In this section, we first review the definition of viscosity solution to \eqref{DCP} and \eqref{Intro:eq4}. Afterwards, we recall several lemmas concerning the interior regularity, Harnack inequality and comparison principle of viscosity solutions for degenerate fully nonlinear equations.

\begin{Definition}
A couple of functions $ u, v \in C^{0}(\overline{B}_{1}) $ is called a viscosity sub-solution [resp. super-solution] of \eqref{DCP}, if for every $ x_{0} \in B_{1} $ we have the following

$\mathrm{i)}$ either for all $ \varphi, \psi \in C^{2}(B_{1}) $ and $ u - \varphi, v- \psi $ attain local maximum [resp. minimum] at $ x_{0} $ and $ D\varphi(x_{0}), D \psi(x_{0}) \neq 0 $, we have
\begin{equation*}
  \left\{
     \begin{aligned}
     & |D\varphi(x_{0})|^{p} F(D^{2}\varphi(x_{0}),x_{0}) \geq v_{+}^{\lambda_{1}}(x_{0})    \\
     &  |D\psi(x_{0})|^{q} G(D^{2}\psi(x_{0}),x_{0}) \geq   u_{+}^{\lambda_{2}}(x_{0}) ,       \\
     \end{aligned}
     \right.
  \Bigg[resp.
  \left\{
     \begin{aligned}
     & |D\varphi(x_{0})|^{p} F(D^{2}\varphi(x_{0}),x_{0}) \leq v_{+}^{\lambda_{1}}(x_{0})    \\
     &  |D\psi(x_{0})|^{q} G(D^{2}\psi(x_{0}),x_{0}) \leq   u_{+}^{\lambda_{2}}(x_{0}) .        \\
     \end{aligned}
     \right. \Bigg]
\end{equation*}

$\mathrm{ii)}$ or there is an open ball $ B_{\delta}(x_{0}) \subset B_{1}$, $ \delta > 0 $ such that $ u \equiv a_{1} $ and $ v \equiv a_{2} $ in $ B_{\delta}(x_{0}) $ and
\begin{equation*}
 (a_{1})_{+}^{\lambda_{2}}, (a_{2})_{+}^{\lambda_{1}}  \leq 0 \ [resp. \ (a_{1})_{+}^{\lambda_{2}}, (a_{2})_{+}^{\lambda_{1}}  \geq 0 ]
\end{equation*}

A couple of functions $ u, v \in C^{0}(\overline{B}_{1}) $ is called a viscosity solution of \eqref{DCP}, if it is both a viscosity sub-solution and super-solution of \eqref{DCP}.
\end{Definition}

\begin{Definition}
Let $ f:B_{1} \times \mathbb{R} \rightarrow  \mathbb{R}    $ be a continuous and bounded function. We say that a function $ u \in C^{0}(B_{1}) $ is a viscosity sub-solution [resp. super-solution] of \eqref{Intro:eq4}, if for every $ x_{0} \in B_{1} $ we have the following

$\mathrm{i)}$ either for all $ \varphi \in C^{2}(B_{1}) $ such that $ u- \varphi $ has a local maximum [resp. minimum] at $ x_{0} $ and $ D\varphi(x_{0}) \neq 0 $, we have
\begin{equation*}
  |D\varphi(x_{0})|^{p} F(D^{2}\varphi(x_{0}),x_{0}) \geq f(|x_{0}|, u(x_{0})) \ \ [\text{resp}. \ |D\varphi(x_{0})|^{p} F(D^{2}\varphi(x_{0}),x_{0}) \leq f(|x_{0}|, u(x_{0}))].
\end{equation*}

$\mathrm{ii)}$ or there is an open ball $ B_{\delta}(x_{0}) \subset B_{1}$, $ \delta > 0 $ such that $ u \equiv c $ in $ B_{\delta}(x_{0}) $ and
$$ f(|x|, c) \leq 0 \ \ [resp. \ f(|x|, c) \geq 0] \ \ \text{for all} \ \ x \in B_{\delta}(x_{0}).       $$

A function $ u \in C^{0}(B_{1}) $ is a viscosity solution of \eqref{Intro:eq4}, if it is both a viscosity sub-solution and super-solution of \eqref{Intro:eq4}.
\end{Definition}

\vspace{1mm}

The following interior gradient H\"{o}lder regularity will be an important tool for our arguments.

\begin{Lemma}[{\bf \cite[Theorem 3.1]{ART15}}]\label{Section2:lem1}
Assume $ u \in C^{0}(B_{1}) $ is a viscosity solution to
\begin{equation*}
  |Du|^{p} F(D^{2}u,x) = f(x)  \ \ \text{in} \ \ B_{1},
\end{equation*}
where $ p > 0 $, $ f\in C^{0}(B_{1}) \cap L^{\infty}(B_{1}) $, and $ F $ satisfies \eqref{1a} and \eqref{2a}. Then $ u \in C^{1,\alpha}_{\mathrm{loc}}(B_{1}) $ for some $ \alpha \in (0,1) $ with the estimate
$$ ||u||_{C^{1,\alpha}(B_{1/2})} \leq C \bigg( ||u||_{L^{\infty}(B_{1})} + ||f||_{L^{\infty}(B_{1})}^{\frac{1}{1+p}}    \bigg),      $$
where $ C= C(p, \lambda, \Lambda, n)   $.
\end{Lemma}

\vspace{1mm}

\begin{remark}
We stress that the H\"{o}lder exponent $ \alpha $ is optimal, that is, $$ \alpha = \min \bigg\{\alpha_{0}^{-}, \frac{1}{1+p}\bigg\}, $$
where $ \alpha_{0} $ is the optimal H\"{o}lder exponent for solutions to constant coefficient, homogeneous equation $ F(D^{2}u) = 0 $.
\end{remark}

\vspace{1mm}

\begin{remark}\label{Sec2:rk2}
Lemma \ref{Section2:lem1} is still true for case $ -1 <p <0 $, see \cite[Theorem 2.1]{BPRT20}, or \cite[Theorem 1.1]{BBLL24}.
\end{remark}

\vspace{1mm}

In order to access improved regularity estimate, we will also assume the following conditions on $ f $.

\vspace{1mm}

\label{Section2:ass1}  {\bf (A6).} We suppose that $ f $ is H\"{o}lder continuous and $ f(0) = 0 $, namely, $$ |f(x)| \leq K|x|^{\widetilde{\lambda}}  $$ holds for some $ K > 0 $ and $ \widetilde{\lambda} \in (0,1) $.

\vspace{1mm}

An improved gradient H\"{o}lder estimate is formulated as follows.

\begin{Lemma}[{\bf \cite[Theorem 1]{T24}}]\label{Section2:lem2}
Let $ u \in C^{0}(B_{1}) $ be a viscosity solution to
\begin{equation*}
  |Du|^{p} F(D^{2}u,x) = f(x)  \ \ \text{in} \ \ B_{1},
\end{equation*}
where $ p > 0 $, and $ F $ satisfies \eqref{1a} and \hyperref[Section1:ass2b]{\bf (A2b)}. Assume \hyperref[Section2:ass1]{\bf (A6)} also holds, then $ u $ is of class $ C^{1,\min\{\alpha_{0}^{-},\frac{1+\widetilde{\lambda}}{1+p} \}} $ at the origin, with the estimate
\begin{equation*}
  |u(x)-u(0)-\nabla u(0)\cdot x| \leq C_{\beta}|x|^{1+\beta}
\end{equation*}
for all $ x \in B_{\frac{1}{4}}(0) $, where
\begin{equation*}
  \beta = \min\bigg\{\alpha_{0}^{-},\frac{1+\widetilde{\lambda}}{1+p} \bigg\}.
\end{equation*}
\end{Lemma}

\begin{remark}
Lemma \ref{Section2:lem2} remains valid for any ball $ B' \Subset B_{1} $. Moreover, it demonstrates that the H\"{o}lder continuity of $ f $ quantitatively influences the regularity of viscosity solutions to degenerate equation, see \cite{T24} for further details.
\end{remark}

The next assertion will be appeared in the analysis of Lemma \ref{Se4:lem1} in Section \ref{Section 3}.
\begin{Lemma}[{\bf \cite[Proposition 1.1]{BD15}}]
\label{Section2:lem3}
  Suppose that, for $ p \in (-1, 0) $, $ u $ is a viscosity solution of
 \begin{equation*}
  |Du|^{p} F(D^{2}u,x) = f(x)  \ \ \text{in} \ \ B_{1},
\end{equation*}
then $ u $ is a viscosity solution of
\begin{equation*}
  F(D^{2}u,x) = f(x) |Du|^{-p}  \ \ \text{in} \ \ B_{1}.
\end{equation*}
\end{Lemma}

We proceed to establish a Harnack inequality for a class of degenerate fully nonlinear equations, which can be readily adapted to the proof of Theorem \ref{Thm5}.

\begin{Theorem}[{\bf \cite[Theorem 1.1]{DFQ10}}]
\label{Pre:Thm2.1}
Let $ u $ be a non-negative viscosity solution to
  \begin{equation*}
  |Du|^{p} F(D^{2}u,x) = f \in C^{0}(B_{1})  \cap L^{\infty}(B_{1}),
  \end{equation*}
where $ 0 \leq p < \infty  $. Then

\begin{equation*}
  \sup_{B_{1/2}} u(x) \leq  C \bigg( \inf_{B_{1/2}} u(x) + ||f||_{L^{\infty}(B_{1})}^{\frac{1}{1+p}}  \bigg).
\end{equation*}
\end{Theorem}

To show Theorem \ref{Thm6}, we need to present several important facts. We first consider a prior estimate to solutions of the Dirichlet problem
\begin{equation}\tag{{\bf D-BVP}}
\label{Se2:D-BVP}
\left\{
     \begin{aligned}
     & |Du|^{p} F(D^{2}u,x) = f(x,u)   \quad \mathrm{in} \ \ B_{1}          \\
     &  u = g   \quad \quad \quad \quad \quad \quad \quad \quad \quad \ \ \mathrm{on}   \ \ \partial B_{1},        \\
     \end{aligned}
     \right.
\end{equation}
where $ 0 \leq p < \infty  $, $ g \in C^{0}(\partial B_{1}) $ and $ f $ satisfies the following condition: for every interval $ K \Subset \mathbb{R} $,
\begin{equation}\label{Se2:eqf1}
    \sup_{B_{1}\times K} |f(x,t)| < \infty.
\end{equation}

Now we establish a priori estimate to solutions of Dirichlet problem \eqref{Se2:D-BVP} when the inhomogeneous term $ f(x,t) $ is subject to appropriate growth conditions as $ |t| \rightarrow \infty $, drawing on ideas from \cite[Lemma 4.1]{BM11} and \cite[Theorem 5.3]{BM12}.  

\begin{Theorem}[{\bf $ L^{\infty}$-bounds}]
\label{Se2:Thm2.2}
Assume that $ F $ fulfills \eqref{1a} and \eqref{2a}. Let $ g \in C^{0}(\partial B_{1}) $ and $ f \in C^{0}(B_{1} \times \mathbb{R}, \mathbb{R}) $ such that \eqref{Se2:eqf1} holds. Suppose that

\vspace{2mm}

$ (\mathrm{i}) $ $ \liminf_{t \rightarrow \infty} \frac{\inf_{B_{1}}f(x,t)}{t^{1+p}}: = \alpha $; \quad $ (\mathrm{ii}) $  $ \liminf_{t \rightarrow -\infty} \frac{\sup_{B_{1}}f(x,t)}{t^{1+p}}: = \beta $

\vspace{2mm}

\noindent for some $ \alpha, \beta \in [0, \infty] $. Then there exists a constant $ C > 0 $, depending on $ f $ and $ g $, such that
\begin{equation*}
  ||u||_{L^{\infty}(B_{1})}  \leq C
\end{equation*}
for any solution $ u \in C^{0}(B_{1}) $ of \eqref{Se2:D-BVP}.
\end{Theorem}

To keep the paper easy to read, we postpone its proof to \hyperref[Appendix A]{Appendix A}.

\vspace{1mm}

The next result is pivotal in establishing $ L^{\infty}$-bounds for viscosity solutions of \eqref{Se2:D-BVP}. We refer the readers to \cite[Lemma 2.3]{1DV21} for the details.

\begin{Lemma}[{\bf Comparison principle}]
\label{Sec2:lemma4}
Suppose that $ F $ satisfies \eqref{1a} and \eqref{2a}. Let $ u_{1} $ and $ u_{2} $ be continuous functions in $ B_{1} $ and let $ f \in C^{0}(\overline{B_{1}})$ fulfilling
\begin{equation*}
  |Du_{1}|^{p} F(D^{2}u_{1},x)  \leq f(x)  \leq |Du_{2}|^{p} F(D^{2}u_{2},x)  \quad    \mathrm{in}  \quad  B_{1}
\end{equation*}
in the viscosity sense, $ 0< p < \infty $, and $ \inf_{B_{1}} f > 0 $ or $ \sup_{B_{1}} f <0 $. If $ u_{1} \geq u_{2} $ on $ \partial B_{1} $, then $ u_{1} \geq u_{2} $ in $ B_{1} $.
\end{Lemma}

The following result is the ``Cutting Lemma" from \cite[Lemma 6]{LS13}(cf. \cite[Lemma 1.7]{SV21}), and it is concerned with the homogeneous degenerate problem:

\begin{Lemma}[{\bf Cutting Lemma}]
\label{lemma25}
Let $ F $ be an operator satisfying \eqref{1a} and $ u $ be a viscosity solution of
\begin{equation*}
  |Du|^{p} F(D^{2}u,x) = 0 \ \ \text{in} \ \ B_{1}
\end{equation*}
with $ p \geq 0 $. Then $ u $ is a viscosity solution of
\begin{equation*}
  F(D^{2}u,x) = 0  \ \ \text{in} \ \ B_{1}.
\end{equation*}
\end{Lemma}

We now turn to establishing a weak comparison principle for degenerate or singular fully nonlinear elliptic systems in a general setting
\begin{equation}\label{Section3:eq9}
\left\{
     \begin{aligned}
     & |Du|^{p} F(D^{2}u,x) - c_{1}(x) Q_{1}(v(x)) =0 \ \ \text{in} \ \ B_{1}   \\
     &  |Dv|^{q} G(D^{2}v,x) - c_{2}(x) Q_{2}(u(x)) =0      \ \ \text{in} \ \ B_{1},        \\
     \end{aligned}
     \right.
\end{equation}
where $ c_{i} $ and $ Q_{i} $, $ i=1,2 $, satisfy the following assumptions:

\vspace{1mm}

\label{Section2:ass2} {\bf (A7) (The continuity and positivity of $ c_{i} $).} $ 0< c_{i}  \in C^{0}(\overline{B}_{1}), i=1,2 $.

\vspace{1mm}

\label{Section2:ass3} {\bf (A8) (The continuity and monotonicity of $ Q_{i} $).} $ Q_{i}: \mathbb{R} \rightarrow  \mathbb{R}, i=1,2 $ are continuous, increasing function, and $ Q_{i}(0) = 0 $.

\begin{Lemma}[{\bf Weak comparison principle}]\label{Sec3:le1}
Suppose that \hyperref[Section1:ass1]{\bf (A1)}, \hyperref[Section1:ass2b]{\bf (A2b)}, \hyperref[Section1:ass3]{\bf (A3)}, \hyperref[Section2:ass2]{\bf (A7)} and \hyperref[Section2:ass3]{\bf (A8)} hold. If $ (u_{1}, v_{1}) $ is a viscosity super-solution and $ (u_{2}, v_{2}) $ is a viscosity sub-solution to \eqref{Section3:eq9} such that $ u_{1} \geq u_{2} $ and $ v_{1} \geq v_{2} $ on $ \partial B_{1}$. Then $ u_{1} \geq u_{2} $ or $ v_{1} \geq v_{2} $ inside $ B_{1} $.
\end{Lemma}
\begin{proof}
The idea of the proof is inspired by \cite[Theorem 2.3]{BV23}. We suppose by contradiction that there exists $ x_{0} \in B_{1} $ such that
\begin{equation}\label{Section3:eq10}
  u_{2}(x_{0}) > u_{1}(x_{0}) \ \ \text{and} \ \ v_{2}(x_{0}) > v_{1}(x_{0}).
\end{equation}
Setting
\begin{equation*}
  M : = \sup_{\overline{B}_{1}} (u_{2} - u_{1}) > 0,
\end{equation*}
and for any small $ \epsilon > 0 $, we define
\begin{equation*}
  \omega_{\epsilon}(x,y) = u_{2}(x)- u_{1}(y) - \frac{1}{2\epsilon}|x-y|^{2} \ \ \text{for} \ \ x,y \in B_{1}.
\end{equation*}
We observe that the maximum of $ \omega_{\epsilon} $ is bigger than $ M $. Let $ (x_{\epsilon}, y_{\epsilon}) \in \overline{B}_{1} \times \overline{B}_{1} $ be such that $ \max_{\overline{B}_{1} \times \overline{B}_{1}}\omega_{\epsilon} = \omega_{\epsilon}(x_{\epsilon}, y_{\epsilon}):= M_{\epsilon} \geq M $. It follows from \cite[Lemma 3.1]{CIL92},
\begin{equation*}
  \lim_{\epsilon \rightarrow 0} \frac{|x_{\epsilon}-y_{\epsilon}|^{2}}{2\epsilon} = 0 \ \ \text{and} \ \ \lim_{\epsilon \rightarrow 0}M_{\epsilon} =  M.
\end{equation*}

Noticed that $ x_{\epsilon}, y_{\epsilon} $ must belong to the interior of $ B_{1} $. In fact, it can be seen that $ u_{2}(z_{0}) - u_{1}(z_{0}) = M > 0 = \sup_{\partial B_{1}} (u_{2}- u_{1})  $, which implies $ x_{\epsilon}, y_{\epsilon} \in B' \Subset B_{1}   $. Since $ u_{1}, u_{2} $ are Lipschitz continuous in $ B' $, by Lemma \ref{Section2:lem1}, we can choose a constant $ L > 0 $ such that
\begin{equation*}
  |u_{1}(x)-u_{1}(y)| + |u_{2}(x) - u_{2}(y)| \leq L|x-y|, \ \ x, y \in B'.
\end{equation*}
Observing
\begin{equation*}
  u_{2}(x_{\epsilon})- u_{1}(x_{\epsilon})  \leq u_{2}(x_{\epsilon}) - u_{1}(y_{\epsilon})  - \frac{1}{2\epsilon}|x_{\epsilon}-y_{\epsilon}|^{2},
\end{equation*}
then it immediately follows that
\begin{equation*}
  |x_{\epsilon}-y_{\epsilon}| \leq 2 \epsilon L.
\end{equation*}
Particularly, we have
\begin{equation*}
  \lim_{\epsilon \rightarrow 0} x_{\epsilon} = \lim_{\epsilon \rightarrow 0} y_{\epsilon}   = z_{0}
\end{equation*}
for some $ z_{0} \in B_{1} $.

In the spirit of \cite[Theorem 3.2]{CIL92}, it infers that there exist matrices $ \mathrm{X}, \mathrm{Y} \in  \mathrm{Sym}(n) $ such that
\begin{equation}\label{Section3:eq11}
\bigg( \frac{x_{\epsilon}-y_{\epsilon}}{\epsilon}, \mathrm{X}  \bigg) \in \overline{J}_{B_{1}}^{2,+}u_{2}(x_{\epsilon}) \ \ \text{and} \ \ \bigg( \frac{x_{\epsilon}-y_{\epsilon}}{\epsilon}, \mathrm{Y}  \bigg) \in \overline{J}_{B_{1}}^{2,-}u_{1}(y_{\epsilon}).
\end{equation}
Moreover,
\begin{equation*}
-\frac{1}{\epsilon} \textbf{I}\rm{d}_{n} \leq
\begin{pmatrix}
\mathrm{X}  &   0   \\
0  &    -\mathrm{Y}
\end{pmatrix}
\leq \frac{1}{\epsilon}
\begin{pmatrix}
\textbf{I}\rm{d}_{n}   &   -\textbf{I}\rm{d}_{n}   \\
-\textbf{I}\rm{d}_{n}  &    \textbf{I}\rm{d}_{n}
\end{pmatrix},
\end{equation*}
which means that
\begin{equation}\label{Section3:eq12}
 \mathrm{X} \leq  \mathrm{Y} \quad \mathrm{and} \quad ||\mathrm{Y}|| \leq \frac{1}{\epsilon}.
\end{equation}

If $ u_{1}, v_{1} $ and $ u_{2}, v_{2} $ are not locally constants around $ y_{\epsilon} $ and $ x_{\epsilon} $, respectively, then we combine \eqref{Section3:eq11}, \eqref{Section3:eq12}, \hyperref[Section1:ass1]{\bf (A1)} and \hyperref[Section1:ass2b]{\bf (A2b)} to get
\begin{align*}
c_{1}(x_{\epsilon})Q_{1}(v_{2}(x_{\epsilon})) & \leq  \bigg( \frac{|x_{\epsilon}-y_{\epsilon}|}{\epsilon} \bigg)^{p} F(\mathrm{X},x_{\epsilon}) \leq \bigg( \frac{|x_{\epsilon}-y_{\epsilon}|}{\epsilon} \bigg)^{p}
F(\mathrm{Y},x_{\epsilon})     \\
& = \bigg( \frac{|x_{\epsilon}-y_{\epsilon}|}{\epsilon} \bigg)^{p} \big[ F(\mathrm{Y},x_{\epsilon}) - F(\mathrm{Y},y_{\epsilon})   \big] + \bigg( \frac{|x_{\epsilon}-y_{\epsilon}|}{\epsilon} \bigg)^{p} F(\mathrm{Y},y_{\epsilon})    \\
& \leq  C\bigg( \frac{|x_{\epsilon}-y_{\epsilon}|}{\epsilon} \bigg)^{p} \omega_{F}(|x_{\epsilon}-y_{\epsilon}|) (1+||\text{Y}||)^{\tau} + c_{1}(y_{\epsilon}) Q_{1}(v_{1}(y_{\epsilon}))     \\
& \leq  C\bigg( \frac{|x_{\epsilon}-y_{\epsilon}|}{\epsilon} \bigg)^{p} \epsilon^{\overline{\alpha}} \bigg(1+\frac{1}{\epsilon}\bigg)^{\tau} + c_{1}(y_{\epsilon}) Q_{1}(v_{1}(y_{\epsilon})),
\end{align*}
then it reads
\begin{align}\label{Section3:eq13}
\begin{split}
 C\bigg( \frac{|x_{\epsilon}-y_{\epsilon}|}{\epsilon} \bigg)^{p} \epsilon^{\overline{\alpha}} \bigg(1+\frac{1}{\epsilon}\bigg)^{\tau} &  + c_{1}(x_{\epsilon}) \big[ Q_{1}(v_{1}(y_{\epsilon})) - Q_{1}(v_{2}(x_{\epsilon})) \big]  \\
& + Q_{1}(v_{1}(y_{\epsilon})) \big[c_{1}(y_{\epsilon}) - c_{1}(x_{\epsilon}) \big]  \geq 0 .
\end{split}
\end{align}
Taking the limit $ \epsilon \rightarrow 0 $ in \eqref{Section3:eq13} and invoking \eqref{Section3:eq10}, it deduces $ c_{1}(z_{0}) \leq 0 $, which is a contradiction to \hyperref[Section2:ass3]{\bf (A7)}.

If $ u_{1} \equiv a_{1}, v_{1} \equiv b_{1} $ in $ B_{\delta}(y_{\epsilon}) $ for some $ \delta > 0 $, and
\begin{equation}\label{Section3:eq14}
  c_{1}(y_{\epsilon})Q_{1}(v_{1}(y_{\epsilon})) \geq 0 \ \ \text{and} \ \ c_{2}(y_{\epsilon})Q_{2}(u_{1}(y_{\epsilon})) \geq 0.
\end{equation}
Similarly, if $ u_{2} \equiv a_{2}, v_{2} \equiv b_{2}  $ in $ B_{\delta}(x_{\epsilon}) $ for some $ \delta > 0 $, and
\begin{equation}\label{Section3:eq15}
  c_{1}(x_{\epsilon})Q_{1}(v_{2}(x_{\epsilon})) \leq 0 \ \ \text{and} \ \ c_{2}(x_{\epsilon})Q_{2}(u_{2}(x_{\epsilon})) \leq 0.
\end{equation}
Taking the limit $ \epsilon \rightarrow 0 $ in   \eqref{Section3:eq14}--\eqref{Section3:eq15}, then we have
\begin{equation*}
  c_{1}(z_{0}) (Q_{1}(b_{2})-Q_{1}(b_{1})) \leq 0 \ \ \text{and} \ \ c_{2}(z_{0}) (Q_{2}(a_{2})-Q_{2}(a_{1})) \leq 0,
\end{equation*}
which means $ c_{i}(z_{0}) \leq 0, i=1,2 $, where we have used the fact $ b_{2} > b_{1} $ and $ a_{2} > a_{1} $. However this is a contradiction to \hyperref[Section2:ass3]{\bf (A7)}.

Consequently, the proof is complete.
\end{proof}

Before proceeding further, we make the following important remarks.

\begin{remark}
The case $ c_{1}(x) = c_{2}(x) \equiv 1 $, $ Q_{1}(v) = v_{+}^{\lambda_{1}} $, $ Q_{2}(u) = u_{+}^{\lambda_{2}} $ and $ p=q \equiv 0 $ corresponds \cite[Lemma 2.1]{AT24}.
\end{remark}

\begin{remark}
The method in Lemma \ref{Sec3:le1} does work for general degenerate or singular elliptic systems with Hamiltonian term
 \begin{equation*}
\left\{
     \begin{aligned}
     & |Du|^{p} F(D^{2}u,x) + H_{1}(x, Du)= c_{1}(x) Q_{1}(v(x)) \ \ \text{in} \ \ B_{1}   \\
     &  |Dv|^{q} G(D^{2}v,x) + H_{2}(x, Dv)= c_{2}(x) Q_{2}(u(x))      \ \ \text{in} \ \ B_{1},         \\
     \end{aligned}
     \right.
\end{equation*}
where $ H_{i}(i=1,2) $ satisfy the following conditions:

\begin{enumerate}[a)]

\item $ H_{1}(x, \widehat{p}) \leq C(1+|\widehat{p}|^{\beta_{1}}) $ and $ H_{2}(x, \widehat{p}) \leq C(1+|\widehat{p}|^{\beta_{2}}) $ for some $ \beta_{1} \in (0, 1+p] $ and $ \beta_{2} \in (0, 1+q] $, $ (x, \widehat{p}) \in \overline{B_{1}} \times \mathbb{R}^{n} $;

\item $ |H_{1}(x, \widehat{p})- H_{1}(y, \widehat{p})| \leq \omega_{1}(|x-y|)(1+|\widehat{p}|^{\beta_{1}}) $ and $ |H_{2}(x, \widehat{p})- H_{2}(y, \widehat{p})| \leq \omega_{2}(|x-y|)(1+|\widehat{p}|^{\beta_{2}}) $, where $ \omega_{i}: [0, \infty)  \rightarrow [0, \infty) $ are continuous functions with $ \omega_{i}(0) =0, i=1,2 $.
\end{enumerate}
\end{remark}

In the end, we present an important lemma involving a particular comparison principle for degenerate fully nonlinear equation, which shall be used in the proof of Theorem \ref{Thm6}.

\begin{Lemma}\label{Se2:lemma7}
Suppose that $ F $ satisfies \eqref{1a}, \eqref{2a} and \eqref{3a}. Assume that $ f_{i}: B_{1} \times \mathbb{R} \rightarrow \mathbb{R}      $ for $ i=1,2 $ are continuous. Let $ u, v \in C^{0}(\overline{B_{1}}) $ satisfy
\begin{equation*}
  |Du|^{p} F(D^{2}u,x) \geq f_{1}(x,u)  \quad  \mathrm{and}  \quad   |Dv|^{p} F(D^{2}v,x)   \leq f_{2}(x,v).
\end{equation*}
Furthermore, assume that either $ f_{1}(x,t) $ or $ f_{2}(x,t) $ is non-decreasing in $ t $, and that $ f_{1}(x,t) > f_{2}(x,t) $ for all $ (x,t) \in B_{1} \times \mathbb{R} $. If $ u \leq v $ on $ \partial B_{1} $, then $ u \leq v  $ in $ B_{1} $.
\end{Lemma}

This lemma can be proved similar to Lemma \ref{Sec3:le1} by using a doubling variable method. For the sake of brevity, we omit its proof.

\vspace{3mm}

\section{Regularity of solutions along the free boundary}\label{Section 3}

Here we derive an improved regularity of solutions to \eqref{DCP} along the free boundary $ \partial \{|(u,v)| >0\} $. With this, we prove a finer gradient control near the free boundary $ \partial \{|(u,v)| >0\} $, as well as a better control for dead-core solutions $ (u,v) $ close to its free boundary. It is meaningful to examine the measure estimate for the free boundary.

We first establish a new flatness estimate, which is different from \cite[Lemma 3.1]{AT24}.

\begin{Lemma}[{\bf Flatness estimate}]\label{Se4:lem1}
Suppose that \hyperref[Section1:ass1]{\bf (A1)}, \hyperref[Section1:ass2a]{\bf (A2a)} and \hyperref[Section1:ass3]{\bf (A3)} hold. For given $ \mu > 0 $, there exists $ \kappa_{\mu} >0 $ such that if $ u, v \in C^{0}(B_{1}) $ satisfies
\begin{equation*}
 u(0)=v(0)=0, \ \ 0\leq u,v \leq 1 \ \ \text{in} \ \ B_{1},
\end{equation*}
and
\begin{equation*}
\left\{
     \begin{aligned}
     &  |Du|^{p} F(D^{2}u,x) - \kappa  (v_{+})^{\lambda_{1}} = 0 \ \ \ \text{in} \ \ B_{1}    \\
     &  |Dv|^{q}G(D^{2}v,x)- (u_{+})^{\lambda_{2}} = 0 \ \ \ \ \  \text{in} \ \ B_{1},       \\
     \end{aligned}
     \right.
\end{equation*}
for $ 0 < \kappa \leq \kappa_{\mu} $, then
\begin{equation*}
  \sup_{B_{1/2}} |(u,v)| \leq \mu.
\end{equation*}
\end{Lemma}

\begin{proof}
We argue by contradiction. Suppose, for $ \mu_{0} > 0 $, there exist the sequence of functions $ \{u_{k}\}_{k}, \{v_{k}\}_{k} $, $ \{F_{k}\}_{k}, \{G_{k}\}_{k}$ such that
$ u_{k}(0) = v_{k}(0) = 0 $, $ 0 \leq u_{k}, v_{k} \leq 1 $ and $ (u_{k},v_{k}) $ satisfies
\begin{equation*}
\left\{
     \begin{aligned}
     &  |Du_{k}|^{p} F_{k}(D^{2}u_{k},x) - \frac{1}{k}  (v_{k})_{+}^{\lambda_{1}} = 0 \ \ \ \text{in} \ \ B_{1}    \\
     &  |Dv_{k}|^{q}G_{k}(D^{2}v_{k},x)- (u_{k})_{+}^{\lambda_{2}} = 0 \ \ \ \ \  \text{in} \ \ B_{1},         \\
     \end{aligned}
     \right.
\end{equation*}
with $ F_{k} $ and $ G_{k}$ still fulfilling \hyperref[Section1:ass1]{\bf (A1)} and \hyperref[Section1:ass2a]{\bf (A2a)}. However,
\begin{equation}\label{sec4:lem1-eq2}
  \sup_{B_{1/2}} |(u_{k},v_{k})| > \mu_{0}.
\end{equation}

Using Lemma \ref{Section2:lem1}, Remark \ref{Sec2:rk2} and the Arzel$ \grave{a}$-Ascoli theorem, up to sub-sequence, $ u_{k} $ and $ v_{k} $ converge locally uniformly to some $ u_{\infty} $ and $ v_{\infty} $ respectively, as $ k\rightarrow \infty $. Clearly, $ u_{\infty}(0) = v_{\infty}(0) =0 $, $ 0 \leq u_{\infty}, v_{\infty} \leq 1 $. By the stability result of viscosity solution, we have that $ (u_{\infty}, v_{\infty}) $ satisfies
\begin{equation}\label{sec4:lem1-eq3}
\left\{
     \begin{aligned}
     &  |Du_{\infty}|^{p} F_{\infty}(D^{2}u_{\infty}, x) = 0  \qquad \qquad \ \text{in} \ \  B_{1/2}     \\
     &  |Dv_{\infty}|^{q}G_{\infty}(D^{2}v_{\infty}, x) = (u_{\infty})_{+}^{\lambda_{2}} \qquad   \text{in} \ \ B_{1/2} ,            \\
     \end{aligned}
     \right.
\end{equation}
where $ F_{k} $ and $ G_{k}$ still satisfy \hyperref[Section1:ass1]{\bf (A1)} and \hyperref[Section1:ass2a]{\bf (A2a)}. In the spirit of Lemma \ref{lemma25}, and Lemma \ref{Section2:lem3}, it follows that
\begin{equation*}
  F_{\infty}(D^{2}u_{\infty}, x) = 0 \ \ \text{in} \ \  B_{1/2}.
\end{equation*}
Then the strong maximum principle \cite[Proposition 4.9]{CC95} implies that $ u_{\infty} \equiv 0 $ in $ B_{1/2}  $. Using \eqref{sec4:lem1-eq3}, Lemma \ref{Section2:lem3} and Lemma \ref{lemma25} again, we obtain
\begin{equation*}
  G_{\infty}(D^{2}v_{\infty}, x) = 0 \ \ \text{in} \ \  B_{1/2},
\end{equation*}
which, together with $ v_{\infty}(0) = 0 $ and strong maximum principle \cite[Proposition 4.9]{CC95}, yield $ v_{\infty} \equiv 0 $ in $ B_{1} $. This contradicts with    \eqref{sec4:lem1-eq2}.
\end{proof}

Now we commence with the proof of Theorem \ref{Thm1}.

\begin{proof}[{\bf Proof of Theorem~\ref{Thm1}}]
We assume, with no loss of generality, that $ x_{0} = 0$. We now take $ \mu = 2^{-{\frac{2}{(1+p)(1+q)-\lambda_{1}\lambda_{2}}}} $ and $ \kappa_{\mu} $ is as in Lemma \ref{Se4:lem1}. Letting
\begin{equation*}
  u_{1}(x) = \widehat{A} u(\widehat{B}x)    \ \  \text{and}   \ \ v_{1}(x) = \widehat{A} v(\widehat{B}x)  \ \  \text{in}  \ \  B_{1},
\end{equation*}
where $ \widehat{A}, \widehat{B} $ are to be determined later, simple calculations yield that $ (u_{1},v_{1}) $ satisfies
\begin{equation}\label{Sec4:sys1}
\left\{
     \begin{aligned}
     &  |Du_{1}|^{p} \widetilde{F}(D^{2}u_{1},x) - \widehat{A}^{1+p-\lambda_{1}}\widehat{B}^{2+p} (v_{1}^{+})^{\lambda_{1}} = 0             \ \ \ \text{in} \ \ B_{1}    \\
     & |Dv_{1}|^{q} \widetilde{G}(D^{2}v_{1},x) - \widehat{A}^{1+q-\lambda_{2} }\widehat{B}^{2+q}(u_{1}^{+})^{\lambda_{2}} =0  \ \ \  \  \text{in} \ \ B_{1}         \\
     \end{aligned}
     \right.
\end{equation}
in the viscosity sense, where
\begin{equation}
\widetilde{F}(D^{2}u_{1},x):= \widehat{A} \widehat{B}^{2} F(\widehat{A}^{-1}\widehat{B}^{-2} D^{2}u_{1}, \widehat{B}x) \ \ \text{and} \ \ \widetilde{G}(D^{2}v_{1},x):= \widehat{A} \widehat{B}^{2} G(\widehat{A}^{-1}\widehat{B}^{-2} D^{2}v_{1}, \widehat{B}x).
\end{equation}
Here we choose $ \widehat{A}, \widehat{B} $ as follows,
\begin{equation*}
     \widehat{A}= \min \bigg\{\kappa_{\mu}^{-\big[(1+p-\lambda_{1})+\frac{(2+p)(\lambda_{2}-1-q)}{2+q}\big]},  1, ||u||_{L^{\infty}(B_{1})}^{-1}, ||v||_{L^{\infty}(B_{1})}^{-1}  \bigg\}
\end{equation*}
and
\begin{equation*}
  \widehat{B} =\widehat{A}^{\frac{\lambda_{2}-1-q}{2+q}}.
\end{equation*}

Then it easily seen that $ u_{1}(0) = v_{1}(0) =0 $ and $ 0 \leq u_{1}, v_{1} \leq 1 $ in $ B_{1} $.
Applying Lemma \ref{Se4:lem1} to $ (u_{1},v_{1}) $, it reads
\begin{equation}\label{Se4:eq1}
 \sup_{B_{1/2}} |(u_{1},v_{1})| \leq 2^{-{\frac{2}{(1+p)(1+q)-\lambda_{1}\lambda_{2}}}}.
\end{equation}
Next we set two new scaled functions
\begin{equation*}
  u_{2}(x) = 2^{{\frac{2}{(1+p)(1+q)-\lambda_{1}\lambda_{2}}}}u_{1}\bigg( \frac{x}{2}  \bigg) \ \ \text{and} \ \  v_{2}(x) = 2^{{\frac{2}{(1+p)(1+q)-\lambda_{1}\lambda_{2}}}}v_{1}\bigg( \frac{x}{2}  \bigg).
\end{equation*}
From (\ref{Se4:eq1}), we have $ 0 \leq u_{2}, v_{2} \leq 1 $ and $ u_{2}(0) = v_{2}(0) = 0 $. Moreover, it is clear that $ (u_{2},v_{2}) $ is a viscosity solution to the variant of system (\ref{Sec4:sys1}). In effect, a straightforward computation yields that $ (u_{2}, v_{2}) $ fulfills
 \begin{equation*}
\left\{
     \begin{aligned}
     &  |Du_{2}|^{p} \widetilde{F}_{0}(D^{2}u_{2},x) - \widehat{A}^{1+p-\lambda_{1}}\widehat{B}^{2+p}\frac{2^{2+p}}{K^{1+\lambda_{1}+p}} (v_{2}^{+})^{\lambda_{1}} = 0             \ \ \ \text{in} \ \ B_{1}    \\
     & |Dv_{2}|^{q} \widetilde{G}_{0}(D^{2}v_{2},x) - \widehat{A}^{1+q-\lambda_{2} }\widehat{B}^{2+q}\frac{2^{2+q}}{K^{1+\lambda_{2}+q}}(u_{2}^{+})^{\lambda_{2}} =0  \ \ \   \text{in} \ \ B_{1}         \\
     \end{aligned}
     \right.
\end{equation*}
in the viscosity sense, where
\begin{equation*}
\left\{
     \begin{aligned}
     & \widetilde{F}_{0}(D^{2}u_{2},x):= \frac{2^{2}}{K} \widetilde{F}(\frac{K}{2^{2}}D^{2}u_{2}, \frac{x}{2});       \\
     & \widetilde{G}_{0}(D^{2}v_{2},x):= \frac{2^{2}}{K} \widetilde{G}(\frac{K}{2^{2}}D^{2}v_{2}, \frac{x}{2});           \\
     &  K:= 2^{{\frac{2}{(1+p)(1+q)-\lambda_{1}\lambda_{2}}}}.
     \end{aligned}
     \right.
\end{equation*}
Then using Lemma \ref{Se4:lem1} again, it infers
\begin{equation*}
  \sup_{B_{1/2}} |(u_{2},v_{2})| \leq 2^{-{\frac{2}{(1+p)(1+q)-\lambda_{1}\lambda_{2}}}},
\end{equation*}
scaling back to $ (u_{1},v_{1}) $, it deduces
\begin{equation*}
  \sup_{B_{1/2^{2}}} |(u_{1},v_{1})| \leq 2^{-2 \cdot{\frac{2}{(1+p)(1+q)-\lambda_{1}\lambda_{2}}}}.
\end{equation*}
Iterating inductively the above reasoning gives the following geometric decay:
\begin{equation}\label{Se4:eq2}
  \sup_{B_{1/2^{i}}} |(u_{1},v_{1})| \leq 2^{-i \cdot{\frac{2}{(1+p)(1+q)-\lambda_{1}\lambda_{2}}}},
\end{equation}
Then, for given $ 0< r <\frac{\widehat{B}}{2} $, there exists $ i \in \mathbb{N} $ such that $ 2^{-(i+1)} <  \frac{r}{\widehat{B}} \leq 2^{-i}  $ and by \eqref{Se4:eq2}, we have
\begin{align*}
\sup_{B_{r}} |(u,v)| \leq \sup_{B_{r/\widehat{B}}} |(u_{1},v_{1})| & \leq  \sup_{B_{2^{-i}}} |(u_{1},v_{1})| \leq 2^{-i \cdot{\frac{2}{(1+p)(1+q)-\lambda_{1}\lambda_{2}}}}   \\
& \leq \big( 2^{-(i+1)}  \big)^{\frac{2}{(1+p)(1+q)-\lambda_{1}\lambda_{2}}} 2^{\frac{2}{(1+p)(1+q)-\lambda_{1}\lambda_{2}}}  \\
& \leq \bigg( \frac{2}{\widehat{B}}  \bigg)^{\frac{2}{(1+p)(1+q)-\lambda_{1}\lambda_{2}}}   r^{\frac{2}{(1+p)(1+q)-\lambda_{1}\lambda_{2}}}
\end{align*}
and the Theorem \ref{Thm1} is shown.
\end{proof}

With Theorem \ref{Thm1}, we can obtain two useful corollaries. The first one is the local growth estimates for first derivative estimate of dead-core solution $ (u,v)$ near the free boundary.
\begin{Corollary}\label{Section4:coro1}
Under the hypotheses of Theorem \ref{Thm1}, and let $ (u,v) \geq 0 $ be a viscosity solution to \eqref{DCP}. Then, for any $ z \in \{ |(u,v)|>0 \} \cap B_{1/2}  $, there holds
  \begin{equation}\label{Se4:eq3}
    |\nabla u(z)| \leq C \bigg(\mathrm{dist}(z, \partial \{ |(u,v)|>0 \} ) \bigg)^{\frac{(1+q)+\lambda_{1}(2+q+\lambda_{2})}{(1+p)(1+q)-\lambda_{1}\lambda_{2}}}
  \end{equation}
  and
  \begin{equation}\label{Se4:eq4}
    |\nabla v(z)| \leq C \bigg(\mathrm{dist}(z, \partial \{ |(u,v)|>0 \} ) \bigg)^{\frac{(1+p)+\lambda_{2}(2+p+\lambda_{1})}{(1+p)(1+q)-\lambda_{1}\lambda_{2}}}.
  \end{equation}
\end{Corollary}

\begin{proof}
We fix $ z \in \{ |(u,v)| >0 \} \cap B_{1/2} $, and denote $ r = \mathrm{dist}(z,\partial \{ |(u,v)| >0 \} ) $. Now we choose a point $ x_{0} \in \partial \{ |(u,v)| >0 \} $ such that $ r = |z-x_{0}| $.

Invoking Theorem \ref{Thm1}, we have
\begin{equation}\label{Se4:eq5}
  \sup_{B_{r}(z)} |(u,v)| \leq \sup_{B_{2r}(x_{0})} |(u,v)| \leq Cr^{\frac{2}{(1+p)(1+q)-\lambda_{1}\lambda_{2}}}.
\end{equation}

Next we define the scaled function $ \widetilde{u},\widetilde{v} $ as follows:
\begin{equation*}
  \widetilde{u}(x) :=    \frac{u(z+rx)}{r^{{\frac{(1+q)(2+p)+\lambda_{1}(2+q)}{(1+p)(1+q)-\lambda_{1}\lambda_{2}}}}} \ \ \text{and} \ \  \widetilde{v}(x) := \frac{v(z+rx)}{r^{^{\frac{(1+p)(2+q)+\lambda_{2}(2+p)}{(1+p)(1+q)-\lambda_{1}\lambda_{2}}}}}.
\end{equation*}
Then a routine calculation reveals that $ (\widetilde{u}, \widetilde{v}) $ satisfies
\begin{equation}\label{Se4:eq6}
\left\{
     \begin{aligned}
     &  |D\widetilde{u}|^{p} \widetilde{F}(D^{2}\widetilde{u},x) = (\widetilde{v}_{+})^{\lambda_{1}}             \ \ \ \text{in} \ \ B_{1}    \\
     &  |D\widetilde{v}|^{q} \widetilde{G}(D^{2}\widetilde{v},x) =(\widetilde{u}_{+})^{\lambda_{2}}   \ \ \  \text{in} \ \ B_{1},       \\
     \end{aligned}
     \right.
\end{equation}
where
\begin{equation*}
  \widetilde{F}(D^{2}\widetilde{u},x) = r^{2-\alpha} F(r^{\alpha-2}D^{2}\widetilde{u}, z+rx), \ \ \widetilde{G}(D^{2}\widetilde{v},x) = r^{2-\beta} G(r^{\beta-2}D^{2}\widetilde{v}, z+rx)
\end{equation*}
and
\begin{equation*}
  \alpha = \frac{(1+q)(2+p)+\lambda_{1}(2+q)}{(1+p)(1+q)-\lambda_{1}\lambda_{2}}, \ \ \beta =  \frac{(1+p)(2+q)+\lambda_{2}(2+p)}{(1+p)(1+q)-\lambda_{1}\lambda_{2}}.
\end{equation*}
Moreover, from \eqref{Se4:eq5}, we know
\begin{equation}\label{Se4:eq7}
  \sup_{B_{r}(z)} |(\widetilde{u},\widetilde{v})| \leq C.
\end{equation}
Finally, we combine the Lemmas \ref{Section2:lem1}, \eqref{Se4:eq6}, \eqref{Se4:eq7} and Remark \ref{Sec2:rk2} to obtain
\begin{equation*}
  |\nabla \widetilde{u}(0)| = \frac{|\nabla u(z)|}{r^{\frac{(1+q)+\lambda_{1}(2+q+\lambda_{2})}{(1+p)(1+q)-\lambda_{1}\lambda_{2}}}} \leq C \ \  \text{and} \ \ |\nabla \widetilde{v}(0)| = \frac{|\nabla v(z)|}{r^{\frac{(1+p)+\lambda_{2}(2+p+\lambda_{1})}{(1+p)(1+q)-\lambda_{1}\lambda_{2}}}} \leq C.
\end{equation*}
Scaling back to $ u, v $, then \eqref{Se4:eq3} and \eqref{Se4:eq4} are proven.
\end{proof}

The second corollary is the finer control for dead-core solutions $ (u,v) $ close its free boundary.

\begin{Corollary}\label{Section4:coro2}
Under the hypotheses of Theorem \ref{Thm1}, and let $ (u,v) \geq 0 $ be a viscosity solution to \eqref{DCP}. Given $ x_{0} \in \{|(u,v)>0|\} \cap B_{1/2}  $, then we have
\begin{equation*}
  |(u(x_{0}), v(x_{0}))| \leq  C \bigg(\mathrm{dist}(x_{0}, \partial\{|(u,v)|>0\})  \bigg)^{\frac{2}{(1+p)(1+q)-\lambda_{1}\lambda_{2}}}.
\end{equation*}
In particular,
\begin{equation*}
  |u(x_{0})| \leq   C \bigg(\mathrm{dist}(x_{0}, \partial\{|(u,v)|>0\})  \bigg)^{\frac{(1+q)(2+p)+\lambda_{1}(2+q)}{(1+p)(1+q)-\lambda_{1}\lambda_{2}}}
\end{equation*}
and
\begin{equation*}
  |v(x_{0})| \leq   C \bigg(\mathrm{dist}(x_{0}, \partial\{|(u,v)|>0\})  \bigg)^{\frac{(1+p)(2+q)+\lambda_{2}(2+p)}{(1+p)(1+q)-\lambda_{1}\lambda_{2}}}.
\end{equation*}
\end{Corollary}

\begin{proof}
For simplicity, we denote $ d:= \text{dist}(x_{0}, \partial\{|(u,v)|>0\}) $. Let $ z_{0} \in \partial\{|(u,v)|>0\} $ be such that $ d= |x_{0}-z_{0}| $. Then applying Theorem \ref{Thm1}, it concludes
\begin{equation*}
  |(u(x_{0}), v(x_{0}))| \leq \sup_{B_{d}(x_{0})} |(u,v)|  \leq \sup_{B_{2d}(z_{0})} |(u,v)|  \leq C d^{\frac{2}{(1+p)(1+q)-\lambda_{1}\lambda_{2}}},
\end{equation*}
which entails the desired result.
\end{proof}

\vspace{3mm}

\section{Radial solution}\label{Section 4}
In this section, we shall find a radial viscosity solution for system \eqref{DCP}. Such a kind of information is vital in deriving the geometric properties (e.g., non-degeneracy) of solutions for many free boundary problems (see \cite{T16, SLR21, SS18, SOS18, SO19, AT24, ALT16, LL22}).

Assume that
\begin{equation*}\label{Section4:eq1}
u(x) = A|x|^{\alpha} \ \ \text{and} \ \ v(x) = B|x|^{\beta},
\end{equation*}
where
\begin{equation*}\label{Section4:eq2}
\alpha := \frac{(1+q)(2+p)+\lambda_{1}(2+q)}{(1+p)(1+q)-\lambda_{1}\lambda_{2}}, \ \  \beta :=  \frac{(1+p)(2+q)+\lambda_{2}(2+p)}{(1+p)(1+q)-\lambda_{1}\lambda_{2}}
\end{equation*}
and $ A, B $ are positive constants to be chosen later.

Simple calculations yield that
\begin{equation}\label{Section4:eq3}
  u_{i} = A \alpha |x|^{\alpha-2} x_{i}, \ \ \  u_{ij} = A \alpha \big( |x|^{\alpha-2} \delta_{ij} + (\alpha-2)|x|^{\alpha-4} x_{i}x_{j}     \big), \ \ i,j=1,2,\cdots,n,
\end{equation}
then by \hyperref[Section1:ass1]{\bf (A1)} and \hyperref[Section1:ass4]{\bf (A4)}, we give
\begin{equation}\label{Section4:eq4}
  F(D^{2}u,x) \leq \Lambda A\alpha (n+\alpha-2)|x|^{\alpha-2}.
\end{equation}
We combine \eqref{Section4:eq3}, \eqref{Section4:eq4} and \eqref{DCP} to obtain
\begin{equation*}
  |Du|^{p} F(D^{2}u,x) \leq \Lambda(n+\alpha-2) (A\alpha)^{1+p} |x|^{(\alpha-1)p+(\alpha-2)}.
\end{equation*}
If $ B $ is such that
\begin{equation}\label{Section4:eq5}
  \Lambda(n+\alpha-2) (A\alpha)^{1+p} = B^{\lambda_{1}},
\end{equation}
then
\begin{equation*}
  |Du|^{p} F(D^{2}u,x) \leq  (B|x|^{\beta})^{\lambda_{1}}= v^{\lambda_{1}}.
\end{equation*}
Analogously,
\begin{equation*}
  |Dv|^{q} G(D^{2}v,x) \leq \Lambda (n+\beta-2) (B\beta)^{1+q} |x|^{(\beta-1)q+(\beta-2)}.
\end{equation*}
If $ A $ is such that
\begin{equation}\label{Section4:eq6}
  \Lambda (n+\beta-2) (B\beta)^{1+q} = A^{\lambda_{2}},
\end{equation}
then
\begin{equation*}
  |Dv|^{q} G(D^{2}v,x) \leq (A|x|^{\alpha})^{\lambda_{2}} = u^{\lambda_{2}}.
\end{equation*}
Combining \eqref{Section4:eq5} and \eqref{Section4:eq6} to obtain the values of $ A $ and $ B $ as follows:
\begin{equation}\label{Section4:eq7}
  A = \big[ \Lambda(n+\beta-2)\beta^{1+q} \big]^{\frac{\lambda_{1}}{\lambda_{1}\lambda_{2}-(1+p)(1+q)}}  \big[ \Lambda(n+\alpha-2)\alpha^{1+p} \big]^{\frac{1+q}{\lambda_{1}\lambda_{2}-(1+p)(1+q)}}
\end{equation}
and
\begin{equation}\label{Section4:eq8}
  B = \big[ \Lambda(n+\alpha-2) \alpha^{1+p} \big]^{\frac{\lambda_{2}}{\lambda_{1}\lambda_{2}-(1+p)(1+q)}} \big[ \Lambda(n+\beta-2) \beta^{1+q} \big]^{\frac{1+p}{\lambda_{1}\lambda_{2}-(1+p)(1+q)}}.
\end{equation}

Thereby, with a careful analysis, one can see that the pair of functions
\begin{equation*}
  u(x) = A|x|^{\alpha}   \ \ \text{and} \ \  v(x)= B|x|^{\beta}
\end{equation*}
are a pair of viscosity super-solutions of system \eqref{DCP}.

Similarly, one has
\begin{equation*}
  u(x) = \widehat{A}|x|^{\alpha}   \ \ \text{and} \ \  v(x)= \widehat{B}|x|^{\beta}
\end{equation*}
are a pair of viscosity sub-solutions of system \eqref{DCP}, where the constants
\begin{equation*}
   \widehat{A} = \big[ \lambda(n+\beta-2)\beta^{1+q} \big]^{\frac{\lambda_{1}}{\lambda_{1}\lambda_{2}-(1+p)(1+q)}}  \big[ \lambda(n+\alpha-2)\alpha^{1+p} \big]^{\frac{1+q}{\lambda_{1}\lambda_{2}-(1+p)(1+q)}}
\end{equation*}
and
\begin{equation*}
   \widehat{B} = \big[ \lambda(n+\alpha-2) \alpha^{1+p} \big]^{\frac{\lambda_{2}}{\lambda_{1}\lambda_{2}-(1+p)(1+q)}} \big[ \lambda(n+\beta-2) \beta^{1+q} \big]^{\frac{1+p}{\lambda_{1}\lambda_{2}-(1+p)(1+q)}}.
\end{equation*}

\vspace{3mm}

    \section{Geometric properties of the free boundary}\label{Section 5}
This section is divided into two subsections. In the first subsection, we prove the non-degeneracy of dead-core solution $ (u,v) $ in virtue of the radial solution obtained in Section \ref{Section 4} and weak comparison principle (Lemma \ref{Sec3:le1}). Finally, in the second subsection, we prove porosity of the free boundary, which may be of own interest.

\subsection{Non-degeneracy of solutions}
With Section \ref{Section 4} and Lemma \ref{Sec3:le1}, we now proceed with the proof of Theorem \ref{Thm2}.

\begin{proof}[{\bf Proof of Theorem~\ref{Thm2}}]
Since $ u $ and $ v $ are continuous, it suffices to show the theorem for the points $ x_{0} \in \{|(u,v)|>0\} \cap B_{1/2} $. Set
\begin{equation*}
  \psi_{1}(x)= c_{1} |x-x_{0}|^{\alpha}, \ \ \psi_{2}(x)= c_{2}|x-x_{0}|^{\beta},
\end{equation*}
for $ \alpha := \frac{(1+q)(2+p)+\lambda_{1}(2+q)}{(1+p)(1+q)-\lambda_{1}\lambda_{2}} $, $ \beta := \frac{(1+p)(2+q)+\lambda_{2}(2+p)}{(1+p)(1+q)-\lambda_{1}\lambda_{2}}$, and $ c_{1} \in (0,A) $, $ c_{2} \in (0,B) $ are to be determined.

A routine calculation reveals that
\begin{equation*}
  (\psi_{1})_{i} = c_{1} \alpha |x-x_{0}|^{\alpha-2}(x_{i}-x_{0,i}), \ \ (\psi_{2})_{i} = c_{2} \beta |x-x_{0}|^{\beta-2} (x_{i}-x_{0,i}), \ \ i,j=1,2,\cdots,n.
\end{equation*}
Further, direct calculations yield that
\begin{equation*}
  (\psi_{1})_{ij}= c_{1} \alpha |x-x_{0}|^{\alpha-2} \delta_{ij} + c_{1}\alpha(\alpha-2) |x-x_{0}|^{\alpha-4}(x_{i}-x_{0,i})(x_{j}-x_{0,j}),
\end{equation*}
and
\begin{equation*}
  (\psi_{2})_{ij}= c_{2} \beta |x-x_{0}|^{\alpha-2} \delta_{ij} + c_{2}\beta(\beta-2) |x-x_{0}|^{\alpha-4}(x_{i}-x_{0,i})(x_{j}-x_{0,j}).
\end{equation*}
Using \hyperref[Section1:ass1]{\bf (A1)}, \hyperref[Section1:ass4]{\bf (A4)} and system \eqref{DCP} to give
\begin{equation*}
  |D\psi_{1}|^{p} F(D^{2}\psi_{1},x) - (\psi_{2})_{+}^{\lambda_{1}} \leq \big[ \Lambda(c_{1}\alpha)^{1+p}(n+\alpha-2) - c_{2}^{\lambda_{1}}   \big]|x-x_{0}|^{\beta\lambda_{1}}
\end{equation*}
and
\begin{equation*}
  |D\psi_{2}|^{q} F(D^{2}\psi_{2},x) - (\psi_{1})_{+}^{\lambda_{2}} \leq \big[ \Lambda(c_{2}\beta)^{1+q}(n+ \beta -2) - c_{1}^{\lambda_{2}}   \big]|x-x_{0}|^{\alpha\lambda_{2}}.
\end{equation*}
Choosing suitable $ c_{1}, c_{2} $ such that
\begin{equation*}
  \Lambda(c_{1}\alpha)^{1+p}(n+\alpha-2) - c_{2}^{\lambda_{1}} \leq 0 \ \ \text{and} \ \ \Lambda(c_{2}\beta)^{1+q}(n+ \beta -2) - c_{1}^{\lambda_{2}} \leq 0,
\end{equation*}
which implies that
\begin{equation*}
  |D\psi_{1}|^{p} F(D^{2}\psi_{1},x) - (\psi_{2})_{+}^{\lambda_{1}} \leq 0 \ \ \text{and} \ \ |D\psi_{2}|^{q} F(D^{2}\psi_{2},x) - (\psi_{1})_{+}^{\lambda_{2}} \leq 0.
\end{equation*}

Next we shall prove the following claim:

{\bf Claim}. There exists $ \xi \in \partial B_{r}(x_{0}) $ such that
\begin{equation*}
  (u(\xi), v(\xi)) \geq (\psi_{1}(\xi), \psi_{2}(\xi)).
\end{equation*}
In other words, $ u(\xi) \geq \psi_{1}(\xi) $ and $ v(\xi) \geq \psi_{2}(\xi) $. Once the claim is established, we have
\begin{equation*}
  \sup_{\overline{B_{r}(x_{0})}} |(u,v)| \geq |(u(\xi),v(\xi))| \geq c r^{\frac{2}{(1+p)(1+q)-\lambda_{1}\lambda_{2}}}.
\end{equation*}
Now it remains to prove the claim. If the conclusion is false, then for any $ \xi \in \partial B_{r}(x_{0}) $,
\begin{equation*}
  (u(\xi), v(\xi)) <  (\psi_{1}(\xi), \psi_{2}(\xi)).
\end{equation*}
Define
\begin{equation*}
\widetilde{\psi_{1}}=\left\{
     \begin{aligned}
     &  \min \{u, \psi_{1}\}, \ \ x \in B_{r}(x_{0})       \\
     &  u, \ \ \ \ \ \ \ x \in B_{r}^{c}(x_{0})       \\
     \end{aligned}
     \right.
\end{equation*}
and
\begin{equation*}
\widetilde{\psi_{2}}=\left\{
     \begin{aligned}
     &  \min \{v, \psi_{2}\}, \ \ x \in B_{r}(x_{0})       \\
     &  v, \ \ \ \ \ \ \ x \in B_{r}^{c}(x_{0}) .      \\
     \end{aligned}
     \right.
\end{equation*}
It is clear that $ \widetilde{\psi_{i}} \in C^{0}(B_{1}), i=1,2 $. Since $ (\psi_{1},\psi_{2}) $ is a viscosity super-solution of \eqref{DCP}, and $ (u,v) $ is also a viscosity super-solution of \eqref{DCP}, then $ (\widetilde{\psi_{1}}, \widetilde{\psi_{2}}) $ also is a super-solution. By Lemma \ref{Sec3:le1}, we obtain
\begin{equation*}
 0 = \widetilde{\psi_{1}}(x_{0}) \geq u(x_{0}) >0 \ \ \text{or} \ \  0 = \widetilde{\psi_{2}}(x_{0}) \geq v(x_{0}) >0,
\end{equation*}
which is a contradiction. Hence, this completes the proof of Theorem \ref{Thm2}.
\end{proof}

As a direct consequence of Theorem \ref{Thm1} and Theorem \ref{Thm2}, we have the following:

\begin{Corollary}\label{Se6:coro1}
Suppose that \hyperref[Section1:ass1]{\bf (A1)}, \hyperref[Section1:ass2b]{\bf (A2b)}, \hyperref[Section1:ass3]{\bf (A3)} and \hyperref[Section1:ass4]{\bf (A4)} hold. Let $ (u,v)  $ be a non-negative viscosity solution to \eqref{DCP}, and $ x_{0} \in \partial \{ |(u,v)|>0 \}     $, then
\begin{equation*}
  \frac{1}{C} r^{\frac{2}{(1+p)(1+q)-\lambda_{1}\lambda_{2}}} \leq  \sup_{B_{r}(x_{0})} |(u,v)|  \leq  C r^{\frac{2}{(1+p)(1+q)-\lambda_{1}\lambda_{2}}},
\end{equation*}
where  $ r >0 $ is small enough and $ C = C(n, \lambda, \Lambda, p, q, \lambda_{1}, \lambda_{2}, ||u||_{L^{\infty}(B_{1})}, ||v||_{L^{\infty}(B_{1})})$ is a positive constant.
\end{Corollary}

\subsection{Porosity of the free boundary}
Once we demonstrate the sharp asymptotic behavior of free boundary problem, it naturally follows that we can derive certain weak geometric characteristics of the phase zone. Hence, we shall prove the porosity of the free boundary $ \partial \{ |(u,v)|>0 \} $ in this subsection.

\begin{Corollary}[{\bf Uniform positive density}]
 Suppose that \hyperref[Section1:ass1]{\bf (A1)}, \hyperref[Section1:ass2b]{\bf (A2b)}, \hyperref[Section1:ass3]{\bf (A3)} and \hyperref[Section1:ass4]{\bf (A4)} hold. Let $  (u,v) \geq 0 $ be a bounded viscosity solution to \eqref{DCP}, and $ y \in \partial \{ |(u,v)|>0\} \cap B_{1/2} $, then
\begin{equation*}
  |B_{\rho}(y) \cap \{|(u,v)|>0\}| \geq c \rho^{n} \ \ \text{for any} \ \  \rho \in (0,1/2),
\end{equation*}
where $ c > 0 $ is a universal constant.
\end{Corollary}
\begin{proof}
Applying Theorem \ref{Thm2}, we have that there exists $ y_{0} \in \partial B_{\rho}(y) \cap \{|u,v|>0\}   $ such that
\begin{equation}\label{Section6.2:eq1}
|(u(y_{0}), v(y_{0}))| \geq c_{1} \rho^{\frac{2}{(1+p)(1+q)-\lambda_{1}\lambda_{2}}}.
\end{equation}
Furthermore, for a small $ \theta > 0 $, we observe
\begin{equation}\label{Section6.2:eq2}
  B_{\theta\rho}(y_{0}) \subset \{ |(u,v)|>0 \}.
\end{equation}
Indeed, if there exists $ \widehat{y} $ such that
$ \widehat{y} \in B_{\theta\rho}(y_{0}) \cap  \partial \{ |(u,v)|>0 \}  $, then from Theorem \ref{Thm1}, it deduces
\begin{equation*}
  c_{1} \rho^{\frac{2}{(1+p)(1+q)-\lambda_{1}\lambda_{2}}} \leq |(u(y_{0}), v(y_{0}))|  \leq  \sup_{B_{\theta\rho}(\widehat{y})} |(u,v)| \leq c_{2} (\theta \rho)^{\frac{2}{(1+p)(1+q)-\lambda_{1}\lambda_{2}}},
\end{equation*}
which is a contradiction, provided we choose
$$ 0< \theta < \bigg( \frac{c_{1}}{c_{2}} \bigg)^{\frac{(1+p)(1+q)-\lambda_{1}\lambda_{2}}{2}}.   $$
Therefore we have shown that \eqref{Section6.2:eq2} holds, and
 \begin{equation*}
   |B_{\rho}(y) \cap \{|(u,v)|>0\}| \geq |B_{\rho}(y) \cap B_{\theta\rho}(y_{0})  | \geq c \rho^{n}.
 \end{equation*}
This completes the proof.
\end{proof}

Before proving the porosity of the free boundary, we recall a concept concerning the porosity of set.

\begin{Definition}[{\bf \cite[Definition 3.8]{PSU12}}]
A set $ V \subset \mathbb{R}^{n} $ is called porous with porosity $ \tau $, if there exists an $ R >0 $ such that $ \forall x \in V $ and $ \forall r \in (0, R)$ there exists $ y \in \mathbb{R}^{n} $ such that
      $$ B_{\tau r}(y)   \subset  B_{r}(x) \setminus V .       $$
\end{Definition}

\begin{Corollary}[{\bf Porosity of the free boundary}]
  Suppose that \hyperref[Section1:ass1]{\bf (A1)}, \hyperref[Section1:ass2b]{\bf (A2b)}, \hyperref[Section1:ass3]{\bf (A3)} and \hyperref[Section1:ass4]{\bf (A4)} hold. Let $  (u,v) \geq 0 $ be a bounded viscosity solution to \eqref{DCP}. Then, there exists a constants $ 0< \epsilon \leq 1 $ such that
  $$ \mathscr{H}^{n-\epsilon}(\partial \{|(u,v)|>0\}\cap B_{1/2}) < \infty .  $$
\end{Corollary}
\begin{proof}
Let $ R > 0 $ and $ x_{0} \in B_{1} $ be such that $ \overline{B_{4R}(x_{0})} \subset B_{1} $. From \cite[Theorem 2.1]{KR97}, it is enough to prove the free boundary $ \partial \{|(u,v)|>0\} \cap B_{R}(x_{0}) $ is a $ \frac{\tau}{2}$-porous set for some $ 0 < \tau <1 $. For this purpose, let $ x \in \partial \{|(u,v)|>0\} \cap B_{R}(x_{0}) $. For any $ r \in (0,R) $, clearly it sees $ \overline{B_{r}(x)} \subset B_{2R}(x_{0}) \subset B_{1} $. Let $ y \in \partial B_{r}(x) $ be such that
\begin{equation*}
  \sup_{\partial B_{r}(x)}|(u,v)| = |(u(y), v(y))|.
\end{equation*}
Applying Theorem \ref{Thm2} above, it deduces
\begin{equation}\label{Section6.2:eq3}
|(u(y), v(y))| \geq c_{1} r^{\frac{2}{(1+p)(1+q)-\lambda_{1}\lambda_{2}}}.
\end{equation}
On the other hand, from Corollary \ref{Section4:coro2}, we obtain
\begin{equation}\label{Section6.2:eq4}
|(u(y), v(y))| \leq c_{2} (d(y))^{\frac{2}{(1+p)(1+q)-\lambda_{1}\lambda_{2}}},
\end{equation}
where $ d(y):=\text{dist}(y, \partial \{|(u,v)|>0\}\cap \overline{B_{2R}(x_{0})})    $.
Combining \eqref{Section6.2:eq3} and \eqref{Section6.2:eq4}, we get
\begin{equation*}
  d(y) \geq \bigg( \frac{c_{1}}{c_{2}}  \bigg)^{\frac{(1+p)(1+q)-\lambda_{1}\lambda_{2}}{2}}r := \tau r,
\end{equation*}
which implies
\begin{equation}\label{Section6.2:eq5}
  B_{\tau r}(y) \subset B_{d(y)}(y) \subset \{|(u,v)|>0 \}.
\end{equation}
Now letting $ z $ be a point on the line segment connecting $ x$ and $ y $ such that $ |z-y| \leq \frac{\tau r}{2} $, then it can easily seen
\begin{equation}\label{Section6.2:eq6}
  B_{\frac{\tau r}{2}}(z) \subset B_{\tau r}(y) \cap B_{r}(x).
\end{equation}
Hence, it follows from \eqref{Section6.2:eq5} and \eqref{Section6.2:eq6} that
\begin{equation*}
  B_{\frac{\tau r}{2}}(z) \subset B_{\tau r}(y) \setminus \partial \{|(u,v)|>0\} \subset  B_{r}(x) \setminus \big(\partial \{|(u,v)|>0\} \cap B_{R}(x_{0})    \big),
\end{equation*}
which yields the desired result.
\end{proof}

We remark that the free boundary $ \partial \{|(u,v)|>0\} $ has Lebesgue measure zero.

\vspace{3mm}

\section{Liouville type results and blow-up analysis}\label{Section 6}
In this section, we give an application of Theorem \ref{Thm1}, namely, Liouville-type result. It shows that the dead-core solution $ (u,v) $ is identical zero in $ \mathbb{R}^{n} $ provided the solution vanishes at origin and has a growth control at infinity.

The first result in this section is formulated as follows.

\begin{Theorem}[{\bf Liouville-type result}]
\label{Section7:Thm1}
Under the hypotheses of Theorem \ref{Thm1}, let $ (u,v) $ be a non-negative viscosity solution of
\begin{equation*}
\left\{
     \begin{aligned}
     &  |Du|^{p} F(D^{2}u,x) = (v_{+})^{\lambda_{1}}\ \ \text{in} \ \ \mathbb{R}^{n}    \\
     &  |Dv|^{q} G(D^{2}v,x) =  (u_{+})^{\lambda_{2}} \ \ \text{in} \ \ \mathbb{R}^{n},       \\
     \end{aligned}
     \right.
\end{equation*}
and $ u(0) = v(0) = 0 $. If
\begin{equation}\label{Section7:eq1}
|(u(x),v(x))| = o\bigg(|x|^{\frac{2}{(1+p)(1+q)-\lambda_{1}\lambda_{2}}}\bigg) \ \ \text{as} \ \ |x|\rightarrow \infty,
\end{equation}
then $ u \equiv v \equiv 0 $ in $ \mathbb{R}^{n} $.
\end{Theorem}

\begin{proof}
We first define the scaled auxiliary functions $ u_{k}, v_{k} $ as follows:
 \begin{equation*}
   u_{k}(x) : =   k^{-\frac{(1+q)(2+p)+\lambda_{1}(2+q)}{(1+p)(1+q)-\lambda_{1}\lambda_{2}}} u(kx), \ \ v_{k}(x) := k^{-\frac{(1+p)(2+q)+\lambda_{2}(2+p)}{(1+p)(1+q)-\lambda_{1}\lambda_{2}}} v(kx).
 \end{equation*}
Direct calculations yield that $ u_{k}(0) = v_{k}(0) = 0$, and $ (u_{k},v_{k}) $ satisfies
\begin{equation*}
\left\{
     \begin{aligned}
     &  |Du_{k}|^{p} F_{k}(D^{2}u_{k},x) = (v_{k})_{+}^{\lambda_{1}}\ \ \text{in} \ \ \mathbb{R}^{n}    \\
     &  |Dv_{k}|^{q} G_{k}(D^{2}v_{k},x) =  (u_{k})_{+}^{\lambda_{2}} \ \ \text{in} \ \ \mathbb{R}^{n},      \\
     \end{aligned}
     \right.
\end{equation*}
in the viscosity sense, where
\begin{equation*}
  F_{k}(D^{2}u_{k},x) = k^{\frac{p(1+q)-\lambda_{1}(2+q+2\lambda_{2})}{(1+p)(1+q)-\lambda_{1}\lambda_{2}}} F\big(k^{\frac{\lambda_{1}(2+q+2\lambda_{2})-p(1+q)}{(1+p)(1+q)-\lambda_{1}\lambda_{2}}}D^{2}u_{k},kx\big),
\end{equation*}
and
\begin{equation*}
  G_{k}(D^{2}v_{k},x) = k^{\frac{q(1+p)-\lambda_{2}(2+p+2\lambda_{1})}{(1+p)(1+q)-\lambda_{1}\lambda_{2}}} G\big(k^{\frac{\lambda_{2}(2+p+2\lambda_{1})-q(1+p)}{(1+p)(1+q)-\lambda_{1}\lambda_{2}}}D^{2}v_{k},kx\big).
\end{equation*}
It is easy to check that $ F_{k} $ and $ G_{k} $ still satisfy \hyperref[Section1:ass1]{\bf (A1)} and \hyperref[Section1:ass2a]{\bf (A2a)}. Moreover, we notice that
\begin{equation}\label{Section7:eq2}
|(u_{k},v_{k})| \rightarrow 0, \ \ \text{as} \ \ k \rightarrow \infty.
\end{equation}
Indeed, let $ x_{k} \in \overline{B_{r}} $ be such that $ \sup_{\overline{B_{r}}} |(u_{k},v_{k})| = |(u_{k}(x_{k}), v_{k}(x_{k}))|  $. Now we consider the following two cases:

Case 1. If $ |k x_{k}| \rightarrow \infty $, as $ k \rightarrow \infty $. By \eqref{Section7:eq1}, we get
\begin{align*}
  |(u_{k}(x_{k}), v_{k}(x_{k}))| & = k^{-\frac{2}{(1+p)(1+q)-\lambda_{1}\lambda_{2}}} |(u(kx_{k}), v(kx_{k}))|    \\
  & \leq |kx_{k}|^{-\frac{2}{(1+p)(1+q)-\lambda_{1}\lambda_{2}}} |(u(kx_{k}), v(kx_{k}))| \rightarrow 0, \ \ \text{as} \ \  k \rightarrow \infty.
\end{align*}

Case 2. If $ |k x_{k}| $ is bounded, then from Theorem \ref{Thm1}, we have
\begin{equation}\label{Section7:eq3}
  |(u_{k}(x_{k}), v_{k}(x_{k}))| \leq C_{k} |x_{k}|^{\frac{2}{(1+p)(1+q)-\lambda_{1}\lambda_{2}}},
\end{equation}
where the constant $ C_{k}= C_{k}(||u_{k}||_{L^{\infty}(B_{1})}, ||v_{k}||_{L^{\infty}(B_{1})}) $, and $ C_{k} \rightarrow 0 $ as $ k \rightarrow \infty $. Scaling back to $ u $, it gives
\begin{equation*}
  |(u(kx_{k}), v(kx_{k}))| \leq C_{k} |kx_{k}|^{\frac{2}{(1+p)(1+q)-\lambda_{1}\lambda_{2}}} \rightarrow 0, \ \ \text{as} \ \  k \rightarrow \infty.
\end{equation*}
Hence, \eqref{Section7:eq2} is shown.

Now if the conclusion of Theorem \ref{Section7:Thm1} is not true, then there exists $ y_{0} \in \mathbb{R}^{n} $ such that $ |(u(y_{0}), v(y_{0}))| >0   $. In the spirit of \eqref{Section7:eq3}, it infers
\begin{equation}\label{Section7:eq4}
  \sup_{B_{1/2}} \frac{|(u_{k}(x), v_{k}(x))|}{|x|^{\frac{2}{(1+p)(1+q)-\lambda_{1}\lambda_{2}}}}  \leq \frac{|(u(y_{0}), v(y_{0})))|}{40|y_{0}|^{\frac{2}{(1+p)(1+q)-\lambda_{1}\lambda_{2}}}}
\end{equation}
for large enough $ k $. Indeed, taking $ k \gg 2|y_{0}| $, from \eqref{Section7:eq4}, one has
\begin{equation*}
\frac{|(u(y_{0}), v(y_{0}))|}{|y_{0}|^{\frac{2}{(1+p)(1+q)-\lambda_{1}\lambda_{2}}}} \leq \sup_{B_{k/2}} \frac{|(u(z), v(z))|}{|z|^{\frac{2}{(1+p)(1+q)-\lambda_{1}\lambda_{2}}}}       =\sup_{B_{1/2}} \frac{|(u_{k}(x), v_{k}(x))|}{|x|^{\frac{2}{(1+p)(1+q)-\lambda_{1}\lambda_{2}}}}  \leq \frac{|(u(y_{0}), v(y_{0}))|}{40|y_{0}|^{\frac{2}{(1+p)(1+q)-\lambda_{1}\lambda_{2}}}},
\end{equation*}
which implies $ |(u(y_{0}), v(y_{0}))| = 0 $. This contradicts with $ |(u(y_{0}), v(y_{0}))| > 0 $. Hence, we finish the proof of theorem.
\end{proof}

An improved Liouville-type result of Theorem \ref{Section7:Thm1} is Theorem \ref{Section7:Thm2}.

\begin{proof}[{\bf Proof of Theorem \ref{Section7:Thm2}}]
We follow the ideas of \cite[Theorem 5.2]{AT24}(see also     \cite[Theorem 5.1]{ALT16} and \cite[Theorem 4.2]{DT20}). Fixing $ R >0 $, we define
\begin{equation}\label{Section7:eq6}
S_{R} = \sup_{\partial B_{R}} |(u,v)| = \sup_{\partial B_{R}} \bigg( u_{+}^{\frac{2}{(1+q)(2+p)+\lambda_{1}(2+q)}} + v_{+}^{\frac{2}{(1+p)(2+q)+\lambda_{2}(2+p)}}    \bigg).
\end{equation}
In the spirit of \eqref{Section7:eq5}, there exist $ 0 < \theta \ll 1 $ and $ R \gg 1 $ such that
\begin{equation}\label{Section7:eq7}
  R^{\frac{-2}{(1+p)(1+q)-\lambda_{1}\lambda_{2}}} S_{R} \leq \theta m,
\end{equation}
where $ m = \min \big\{A^{\frac{2}{(1+q)(2+p)+\lambda_{1}(2+q)}},   B^{\frac{2}{(1+p)(2+q)+\lambda_{2}(2+p)}} \big\} $.

Now we consider the following two auxiliary functions $ \widetilde{u}(x), \widetilde{v}(x) $:
\begin{equation*}
  \widetilde{u}(x) = A(|x|-\rho)_{+}^{\frac{(1+q)(2+p)+\lambda_{1}(2+q)}{(1+p)(1+q)-\lambda_{1}\lambda_{2}}} \ \ \text{and}  \ \ \widetilde{v}(x) = B(|x|-\rho)_{+}^{\frac{(1+p)(2+q)+\lambda_{2}(2+p)}{(1+p)(1+q)-\lambda_{1}\lambda_{2}}}
\end{equation*}
for some $ \rho >0 $. As the previous analysis in Section \ref{Section 4}, one has the dead-core domain $ B_{\rho} = \{|(\widetilde{u},\widetilde{v})|=0\} $. In order to use Lemma \ref{Sec3:le1}, for $ R \gg 1$, we set
\begin{equation}\label{Section7:eq8}
  \rho := R- \bigg( \frac{S_{R}}{m} \bigg)^{\frac{(1+p)(1+q)-\lambda_{1}\lambda_{2}}{2}} \geq \bigg(1-\theta^{\frac{(1+p)(1+q)-\lambda_{1}\lambda_{2}}{2}}\bigg)R,
\end{equation}
where we have used \eqref{Section7:eq7}. Then we combine \eqref{Section7:eq6} and \eqref{Section7:eq8} to obtain
\begin{equation*}
  \widetilde{u}(x) = A \bigg( \frac{S_{R}}{m} \bigg)^{\frac{(1+q)(2+p)+\lambda_{1}(2+q)}{2}}  \geq \sup_{\partial B_{R}} u,
\end{equation*}
and
\begin{equation*}
  \widetilde{v}(x) = B \bigg( \frac{S_{R}}{m} \bigg)^{\frac{(1+p)(2+q)+\lambda_{2}(2+p)}{2} } \geq \sup_{\partial B_{R}} v
\end{equation*}
for any $ x \in \partial B_{R} $. By Lemma \ref{Sec3:le1}, we deduce
\begin{equation}\label{Section7:eq9}
\widetilde{u} \geq u \ \ \text{or} \ \ \widetilde{v} \geq v \ \ \text{in} \ \ B_{R}.
\end{equation}
Thus, taking $ R \rightarrow \infty $ and using \eqref{Section7:eq8} and \eqref{Section7:eq9}, it follows that $ u(x) v(x) = 0 $ in $ \mathbb{R}^{n} $, and so
\begin{equation}\label{Section7:eq10}
  \{ u >0 \}  \subset \{ v=0 \} \ \ \text{and} \ \ \{ v >0 \}  \subset \{ u=0 \}.
\end{equation}
If there exists $ z_{0} \in \mathbb{R}^{n} $ such that $ u(z_{0}) > 0 $. Using the first part of \eqref{Section7:eq10} and the continuity of $ u $, one has $ v = 0 $ in some neighbourhood of $ z_{0} $. From the second equation of \eqref{Section7:eqq1}, we infer $ u = 0 $ in some neighbourhood of $ z_{0} $, which contradicts with $ u(z_{0}) > 0 $. Thereby, $ u \equiv 0 $ holds. Similarly, we can also show $ v \equiv 0 $.
\end{proof}

\begin{remark}
We remark that the inequality in \eqref{Section7:eq5} is sharp. In fact, when
\begin{equation*}
  u(x) = A |x|^{\frac{(1+q)(2+p)+\lambda_{1}(2+q)}{(1+p)(1+q)-\lambda_{1}\lambda_{2}}} \ \ \text{and} \ \ v(x) = B |x|^{\frac{(1+p)(2+q)+\lambda_{2}(2+p)}{(1+p)(1+q)-\lambda_{1}\lambda_{2}}},
\end{equation*}
\eqref{Section7:eq5} becomes an equality, but neither $ u $ nor $ v $ is identically $ 0 $.
\end{remark}

\begin{remark}
We highlight that the proof of Theorem \ref{Section7:Thm2} does not work for $ -1 < p, q < 0 $.
\end{remark}

In the next theorem, we shall study the blow-up analysis over free boundary points. Let $ (u,v) $ be  solution to dead-core problem \eqref{DCP} and $ z_{0} \in \partial \{ |(u,v)|>0 \} \cap B_{1/2} $. For each $ \tau >0 $, we define a blow-up sequence $ |(u_{\tau}, v_{\tau})|: B_{\frac{1}{2\tau}} \rightarrow \mathbb{R}  $ as
\begin{equation*}
   |(u_{\tau}, v_{\tau})| = \frac{|(u(z_{0}+\tau x),v(z_{0}+\tau x))|}{\tau^{\frac{2}{(1+p)(1+q)-\lambda_{1}\lambda_{2}}}},
\end{equation*}
that is to say,
\begin{equation*}
  u_{\tau}(x) = \frac{u(z_{0}+\tau x)}{\tau^{\frac{(1+q)(2+p)+\lambda_{1}(2+q)}{(1+p)(1+q)-\lambda_{1}\lambda_{2}}}} \ \ \text{and} \ \ v_{\tau}(x) = \frac{v(z_{0}+\tau x)}{\tau^{\frac{(1+p)(2+q)+\lambda_{2}(2+p)}{(1+p)(1+q)-\lambda_{1}\lambda_{2}}}}.
\end{equation*}
The next result analyses the limiting profile for any blow-up sequence.

\vspace{1mm}

\begin{Theorem}[{\bf Blow-up limit}]
Let $ z_{0} \in \partial \{ |(u,v)|>0 \} \cap B_{1/2} $ be a free boundary point, and a blow-up sequence $ |(u_{\tau}, v_{\tau})| $. Then, up to a subsequence,
\begin{equation*}
  (u_{\tau}, v_{\tau})  \rightarrow  (u_{0}, v_{0}) \ \ \text{and} \ \ (Du_{\tau}, Dv_{\tau})  \rightarrow  (Du_{0}, Dv_{0}) \ \text{locally uniformly in $ B_{1}$}.
\end{equation*}
Furthermore, $ (u_{0}, v_{0}) $ is a non-negative viscosity solution to the following global degenerate elliptic system:
 \begin{equation}\label{Section7:eq11}
\left\{
     \begin{aligned}
     & |Du_{0}|^{p} F_{0}(D^{2}u_{0},x) = (v_{0})_{+}^{\lambda_{1}} \ \ \text{in} \ \  \mathbb{R}^{n}    \\
     &  |Dv_{0}|^{q} G_{0}(D^{2}v_{0},x) =  (u_{0})_{+}^{\lambda_{2}} \ \ \text{in} \ \ \mathbb{R}^{n},        \\
     \end{aligned}
     \right.
\end{equation}
where $ F_{0}, G_{0} $ are uniformly elliptic operator. In the end, $ 0 \in \partial \{|(u_{0},v_{0})|>0\} $.
\end{Theorem}
\begin{proof}
A straightforward calculation yields that $ (u_{\tau}, v_{\tau}) $ fulfills
 \begin{equation}\label{Section7:eq12}
\left\{
     \begin{aligned}
     & |Du_{\tau}|^{p} F_{\tau}(D^{2}u_{\tau},x) = (v_{\tau}(x))_{+}^{\lambda_{1}} \ \ \text{in} \ \  B_{\frac{1}{2\tau}}    \\
     &  |Dv_{\tau}|^{q} G_{\tau}(D^{2}v_{\tau},x) =  (u_{\tau}(x))_{+}^{\lambda_{2}} \ \ \text{in} \ \ B_{\frac{1}{2\tau}},      \\
     \end{aligned}
     \right.
\end{equation}
in the viscosity sense, where
\begin{equation*}
  F_{\tau}(D^{2}u_{\tau},x) = \tau^{2-\alpha} F(\tau^{\alpha-2}D^{2}u_{\tau}, z_{0}+\tau x), \ \ G_{\tau}(D^{2}v_{\tau},x) = \tau^{2-\beta} G(\tau^{\beta-2}D^{2}u_{\tau}, z_{0}+\tau x).
\end{equation*}
From Theorem \ref{Thm1}, it follows that
\begin{equation}\label{Section7:eq13}
|u_{\tau}(x)| \leq C |x|^{\frac{(1+q)(2+p)+\lambda_{1}(2+q)}{(1+p)(1+q)-\lambda_{1}\lambda_{2}}} \ \ \text{and} \ \ |v_{\tau}(x)| \leq C |x|^{\frac{(1+p)(2+q)+\lambda_{2}(2+p)}{(1+p)(1+q)-\lambda_{1}\lambda_{2}}}
\end{equation}
for any $ x \in B_{\frac{1}{2\tau}} $.

Now combining \eqref{Section7:eq13}, Lemma \ref{Section2:lem1} and Arzel$ \grave{a}$-Ascoli theorem, then, up to a sequence, we have
\begin{equation*}
  (u_{\tau}, v_{\tau})  \rightarrow  (u_{0}, v_{0}) \ \ \text{and} \ \ (Du_{\tau}, Dv_{\tau})  \rightarrow  (Du_{0}, Dv_{0}) \ \text{locally uniformly in $ B_{1}$}.
\end{equation*}
Taking the limit $ \tau \rightarrow 0 $ in \eqref{Section7:eq12} and using the stability result of viscosity solution, it infers \eqref{Section7:eq11} holds. Finally, since $ 0 \in \partial \{ |(u_{\tau},v_{\tau})|>0 \} $, then $ 0 \in \partial \{|(u,v)|>0\} $ holds clearly.
\end{proof}

An immediate consequence of Corollary \ref{Se6:coro1} is the following:

\begin{Theorem}
Let $ (u_{0}, v_{0}) $ be a blow-up solution at $ z_{0} \in \partial \{ |(u,v)|>0 \}  $. Then, there exists a positive constant $ C=C( n, \lambda, \Lambda, p, q, \lambda_{1}, \lambda_{2}, ||u||_{L^{\infty}(B_{1})}, ||v||_{L^{\infty}(B_{1})}) $ such that
\begin{equation*}
  \frac{1}{C} \leq \liminf_{|x|\rightarrow \infty} \bigg( |x|^{-\frac{2}{(1+p)(1+q)-\lambda_{1}\lambda_{2}}}  |(u_{0}, v_{0})|  \bigg)  \leq  \limsup_{|x|\rightarrow \infty} \bigg( |x|^{-\frac{2}{(1+p)(1+q)-\lambda_{1}\lambda_{2}}}  |(u_{0}, v_{0})|  \bigg) \leq C.
\end{equation*}
\end{Theorem}

\vspace{3mm}

  \section{Proof of Theorem \ref{Thm5}, Corollary \ref{Coro1} and applications}\label{Section 7}
In this section, we give an alternative proof of Theorem~\ref{Thm5} using the Harnack inequality rather than flatness estimate. Subsequently, the proof of Corollary \ref{Coro1} and some applications will be given.

\begin{proof}[{\bf Proof of Theorem~\ref{Thm5}}]
he idea of the proof is inspired by \cite[Theorem 1]{JSNS24} and references therein. Without loss of generality, we assume $ x_{0} = 0 $. We want to show that the existence of a $ C > 0 $, such that for all $ j \in \mathbb{N} $, there holds
\begin{equation}\label{Section9:eq1}
  \theta_{j+1} \leq  \max \big\{ C 2^{-\beta(j+1)}, 2^{-\beta}\theta_{j}     \big\},
\end{equation}
where
\begin{equation}\label{Section9:eq2}
  \theta_{j} := \sup_{B_{2^{-j}}} u    \  \   \text{and} \  \ \beta := \frac{2+p+\alpha}{1+p-\mu}.
\end{equation}
We argue by contradiction, if \eqref{Section9:eq1} does not hold, i.e., for each $ k \in \mathbb{N} $, there exists $ j_{k} \in \mathbb{N} $ such that
\begin{equation}\label{Section9:eq3}
  \theta_{j_{k}+1} >  \max \big\{ k 2^{-\beta(j_{k}+1)}, 2^{-\beta}\theta_{j_{k}}     \big\}.
\end{equation}

Now we define a rescaled function $ v_{k}: B_{1} \rightarrow \mathbb{R} $ as follows:
\begin{equation*}
  v_{k}(x) = \frac{u(2^{-j_{k}}x)}{\theta_{j_{k}+1}},
\end{equation*}
then simple calculations yield that
\begin{equation}\label{Section9:eq4}
  0 \leq v_{k}(x) \leq 2^{\beta}; \ \  v_{k}(0) = 0; \ \ \sup_{\overline{B}_{1/2}} v_{k}(x) = 1.
\end{equation}
Furthermore, $ v_{k} $ is a viscosity solution to
\begin{equation}\label{Section9:eq5}
  |Dv_{k}|^{p} \widetilde{F}(D^{2}v_{k},x) = \widetilde{f}(|x|, v_{k}(x)) \ \ \text{in} \ \ B_{1},
\end{equation}
where
\begin{equation*}
 \widetilde{F}(D^{2}v_{k},x) = 2^{-2j_{k}} \theta_{j_{k+1}}^{-1} F(2^{2j_{k}} \theta_{j_{k+1}}, 2^{-j_{k}}x)
\end{equation*}
and
\begin{equation*}
  \widetilde{f}(|x|, v_{k}(x)) = 2^{-(p+2)j_{k}} \theta_{j_{k+1}}^{-(p+1)} f(2^{-j_{k}}|x|, \theta_{j_{k}+1}v_{k}(x)).
\end{equation*}
Here the operator $  \widetilde{F} $ still satisfies \eqref{1a} and \eqref{2a}.

Using \hyperref[Section1:ass5]{\bf (A5)}, we have
\begin{equation}\label{Section9:eq6}
  ||\widetilde{f}||_{L^{\infty}(B_{1})}  \leq \frac{C_{n}||f_{0}||_{L^{\infty}(B_{1})}}{k^{p+1-\mu}2^{(p+2+\alpha)(2+j_{k})}},
\end{equation}
then from \eqref{Section9:eq4}, \eqref{Section9:eq6}, and the Harnack inequality, Theorem \ref{Pre:Thm2.1}, it concludes
\begin{align*}
  1= \sup_{B_{1/2}} v_{k}(x)  & \leq C \bigg( \inf_{B_{1/2}} v_{k}(x)   + ||\widetilde{f}||_{L^{\infty}(B_{1})}^{\frac{1}{1+p}}
   \bigg)    \\
  &  \leq C \bigg(\frac{C_{n}||f_{0}||_{L^{\infty}(B_{1})}}{k^{p+1-\mu}2^{(p+2+\alpha)(2+j_{k})}}    \bigg)^{\frac{1}{1+p}} \rightarrow 0, \ \ \text{as} \ \ k \rightarrow \infty,
\end{align*}
which is a contradiction. Hence \eqref{Section9:eq1} holds.

For given $ r \in (0,1) $, then there exists $ j \in \mathbb{N} $ such that $ 2^{-j-1} \leq r  < 2^{-j}$. In the spirit of \eqref{Section9:eq1}, it follows that
\begin{equation*}
  \sup_{B_{r}} u \leq \sup_{B_{2^{-j}}} u \leq C 2^{-\beta j} = C 2^{-\beta(1+j)} 2^{\beta} \leq C r^{\beta} = Cr^{\frac{2+p+\alpha}{1+p-\mu}}.
\end{equation*}
We finish the proof of Theorem \ref{Thm5}.
\end{proof}

\begin{proof}[{\bf Proof of Corollary~\ref{Coro1}}]
The proof is analogous to the proof of Corollary \ref{Section4:coro1}, for the sake of completeness, we present it here. Let $ x_{0} \in \{u>0\} \cap B_{1/2}  $ be a point such that $ f(x,t) \simeq |x-\gamma(x_{0})|^{\alpha} t_{+}^{\mu} $, where $ \gamma:= \gamma(x_{0}) \in \partial \{u>0\} $ fulfills
\begin{equation*}
  |x_{0}-\gamma(x_{0})| = \text{dist}(x_{0},\partial \{u>0\}):= d.
\end{equation*}

Invoking Theorem \ref{Thm5}, we obtain
\begin{equation}\label{Section9:eq7}
  \sup_{B_{d}(x_{0})} u \leq \sup_{B_{2d}(\gamma)} u \leq C d^{\frac{2+p+\alpha}{1+p-\mu}}.
\end{equation}

Next we consider an auxiliary function $ v: B_{1} \rightarrow \mathbb{R} $ as follows:
\begin{equation*}
  v(x) = \frac{u(x_{0}+dx)}{d^{\frac{2+p+\alpha}{1+p-\mu}}}.
\end{equation*}
Direct computations yield that $ v $ is a viscosity solution to
\begin{equation*}
  |Dv|^{p} \widetilde{F}(D^{2}v,x) =  \widetilde{f}(|x|,v(x)),
\end{equation*}
where
\begin{equation*}
  \widetilde{F}(D^{2}v,x) = d^{-\frac{\alpha+2\mu-p}{1+p-\mu}} F(d^{\frac{\alpha+2\mu-p}{1+p-\mu}}D^{2}v(x), x_{0}+dx)
\end{equation*}
and
\begin{equation*}
  \widetilde{f}(|x|,v(x)):= d^{-\frac{\alpha(1+p)+\mu(2+p)}{1+p-\mu}} f(|x_{0}+dx|,u(x_{0}+dx)).
\end{equation*}
Here the operator $  \widetilde{F} $ still satisfies \eqref{1a} and \eqref{2a}.
Note that
\begin{equation*}
  \widetilde{f}(|x|,v(x)) \simeq  d^{-\frac{\alpha(1+p)+\mu(2+p)}{1+p-\mu}} d^{\alpha+\frac{\mu(2+p+\alpha)}{1+p-\mu}} |x|^{\alpha} v_{+}^{\mu} =  |x|^{\alpha} v_{+}^{\mu},
\end{equation*}
then from \eqref{Section9:eq7}, Lemma \ref{Section2:lem1} and Remark \ref{Sec2:rk2}, we have
\begin{equation*}
  \frac{|\nabla u(x_{0})|}{d^{\frac{1+\alpha+\mu}{1+p-\mu}}} = |\nabla v(0)| \leq \sup_{B_{1/2}} |\nabla v|  \leq C \big( \sup_{B_{1}} v + ||\widetilde{f}||_{L^{\infty}(B_{1})}^{\frac{1}{1+p}}  \big) \leq C,
\end{equation*}
which gives
\begin{equation*}
  |\nabla u(x_{0})| \leq C d^{\frac{1+\alpha+\mu}{1+p-\mu}}.
\end{equation*}\end{proof}

We conclude this section by giving some applications, i.e., Liouville-type results. The proof is just the same as that of Theorems \ref{Section7:Thm2} and \ref{Section7:Thm1}, and is omitted here.

\begin{Theorem}\label{Section8:thm1}
Under the hypotheses of Theorem \ref{Thm5}, let $ u $ be a non-negative bounded entire viscosity solution to
\begin{equation*}
  |Du|^{p} F(D^{2}u, x) = |x|^{\alpha} (u)_{+}^{\mu} \ \ \text{in} \ \ \mathbb{R}^{n}
\end{equation*}
with $ u(0) = 0 $. Assume that
\begin{equation*}
  \lim_{|x|\rightarrow \infty} \frac{u(x)}{|x|^{\frac{2+p+\alpha}{1+p-\mu}}} = 0,
\end{equation*}
then $ u \equiv 0 $ in $ \mathbb{R}^{n} $.
\end{Theorem}

\begin{Theorem}\label{Section8:thm2}
 Under the hypotheses of Theorem \ref{Thm5}, and \eqref{3a} also holds. Let $ u $ be a non-negative bounded entire viscosity solution to
\begin{equation*}
  |Du|^{p} F(D^{2}u, x) = |x|^{\alpha} (u)_{+}^{\mu} \ \ \text{in} \ \ \mathbb{R}^{n}.
\end{equation*}
Suppose that
\begin{equation*}
  \limsup_{|x|\rightarrow \infty}  \frac{u(x)}{|x|^{\frac{2+p+\alpha}{1+p-\mu}}} < C_{1}(n, p, \mu, \alpha, \Lambda),
\end{equation*}
where $ C_{1}(n, p, \mu, \alpha, \Lambda) $ is the constant in Example \ref{Intro:ex3} of Section \ref{Section9},
then $ u \equiv 0 $ in $ \mathbb{R}^{n} $.
\end{Theorem}

\vspace{3mm}

\section{Proof of Theorem \ref{Thm6} and Theorem \ref{Thm7}}\label{Section 8}

In this section, we will give the proof of Theorem \ref{Thm6} and Theorem \ref{Thm7}.

\begin{proof}[{\bf Proof of Theorem~\ref{Thm6}}]
For simplicity, here we only consider the special case $ f(|x|,u(x)) = |x|^{\alpha}(u)_{+}^{\mu} $, as the general case is similar. For each $ \epsilon >0 $, we first consider the following penalized problem
\begin{equation}\tag{$\mathbf{P}_{\epsilon}$}
\label{P1}
\left\{
     \begin{aligned}
     & |Du_{\epsilon}|^{p} F(D^{2}u_{\epsilon},x) = |x|^{\alpha}(u_{\epsilon})_{+}^{\mu}  +  \epsilon        \quad \ \mathrm{in} \ \   B_{1}         \\
     &  u_{\epsilon} =  u   \qquad \qquad \qquad \qquad \qquad \qquad  \ \  \   \mathrm{on} \ \ \partial B_{1}.                     \\
     \end{aligned}
     \right.
\end{equation}
Note that for each fixed $ \epsilon > 0 $, the existence of such a solution $ u_{\epsilon} $ to \eqref{P1} is guaranteed by the classial Perron's method. For simplicity, we denote $ f_{\epsilon}(x,t):= |x|^{\alpha}(u_{\epsilon})_{+}^{\mu}  +  \epsilon $. Clearly, $ f_{\epsilon} $ satisfies the condition $ (\mathrm{i})$ and $(\mathrm{ii})$ in Theorem \ref{Se2:Thm2.2}, hence, one can see that $ u_{\epsilon} $ is bounded. Observed that
\begin{equation}\label{Se81:1}
  H_{1} \leq u_{\epsilon_{1}} \leq u_{\epsilon_{2}} \leq \cdots \leq u_{\epsilon_{j}} \leq \cdots \leq u \leq H_{2}
\end{equation}
for every subsequence $ \{ \epsilon_{j}\}_{j \in \mathbb{N}} \subset [0,1] $ and $ \epsilon_{j} \rightarrow 0^{+}$, where $ H_{2} \in C^{0}(\overline{B_{1}}) $ satisfies
\begin{equation*}
\left\{
     \begin{aligned}
     & |DH_{2}|^{p} F(D^{2}H_{2},x) = 0   \quad \mathrm{in} \ \ B_{1}         \\
     &  H_{2}   =  u   \qquad \qquad \qquad \qquad \ \mathrm{on} \ \ \partial B_{1}                    \\
     \end{aligned}
     \right.
\end{equation*}
in the viscosity sense, and $ H_{1} \in C^{0}(\overline{B_{1}}) $ fulfills in the viscosity sense
\begin{equation*}
\left\{
     \begin{aligned}
     & |DH_{1}|^{p} F(D^{2}H_{1},x) = |x|^{\alpha}(H_{1})_{+}^{\mu} + 1   \quad \mathrm{in} \ \ B_{1}         \\
     &  H_{1}   =  u   \qquad \qquad \qquad \qquad \qquad \qquad \ \  \mathrm{on} \ \ \partial B_{1}.                          \\
     \end{aligned}
     \right.
\end{equation*}
Next we will establish a non-degeneracy estimate of solution $ u_{\epsilon} $ to \eqref{P1}(uniformly in $ \epsilon$). To achieve this, by continuity, without loss of generality, we assume that $ x_{0} = 0 \in  \{ u > 0 \} \cap B_{1}   $. Then we consider the radial function
\begin{equation*}
  \Phi(x):= \bigg[\frac{(1+p-\mu)^{2+p}}{\Lambda(2+p+\alpha)^{1+p}[n(1+p-\mu)+(2\mu+\alpha-p)]}            \bigg]^{\frac{1}{1+p-\mu}} |x|^{{\frac{2+p+\alpha}{1+p-\mu}}}.
\end{equation*}
By the analysis in Example \ref{Intro:ex3} of Section \ref{Section9}, it can be seen that the function $ \Phi $ is a viscosity super-solution to
\begin{equation*}
  |D \Phi|^{p} F(D^{2}\Phi,x) = |x|^{\alpha} \Phi_{+}^{\mu}  \quad \mathrm{in}  \quad  B_{1}.
\end{equation*}
We claim that there exists $ z_{r} \in \partial B_{r} $ such that $ u_{\epsilon}(z_{r}) > \Phi(z_{r}) $. If suppose not, by comparison principle(Lemma \ref{Se2:lemma7}), $ u_{\epsilon} \leq \Phi $ in $ B_{r} $. However, this leads a contradiction since $ u_{\epsilon}(0) > 0 = \Phi(0) $.

Consequently, we have the following estimate
\begin{align}\label{Se81:2}
\begin{split}
\sup_{\partial B_{r}} u_{\epsilon} & \geq u_{\epsilon}(z_{r})   \\
& > \Phi(z_{r})= \bigg[\frac{(1+p-\mu)^{2+p}}{\Lambda(2+p+\alpha)^{1+p}[n(1+p-\mu)+(2\mu+\alpha-p)]}            \bigg]^{\frac{1}{1+p-\mu}} r^{{\frac{2+p+\alpha}{1+p-\mu}}}.
\end{split}
\end{align}

Finally, using \eqref{Se81:1}, it gives $ u_{\epsilon_{j}} \leq u $. Then taking the supremum and using \eqref{Se81:2}, we obtain
\begin{equation*}
  \sup_{\partial B_{r}} u \geq \sup_{\partial B_{r}} u_{\epsilon} \geq \bigg[\frac{(1+p-\mu)^{2+p}}{\Lambda(2+p+\alpha)^{1+p}[n(1+p-\mu)+(2\mu+\alpha-p)]}            \bigg]^{\frac{1}{1+p-\mu}} r^{{\frac{2+p+\alpha}{1+p-\mu}}}.
\end{equation*}
\end{proof}

\begin{proof}[{\bf Proof of Theorem~\ref{Thm7}}] The idea of the proof is inspired by \cite[Theorem 5]{JSNS24}. Consider an open set $ B^{+} = \{ x \in B_{1}| u(x) >0  \} $. If $ B^{+} $ is nonempty, we get that
\begin{equation*}
  |Du|^{p} F(D^{2}u,x) \geq 0 \quad \text{in} \quad B^{+}.
\end{equation*}
From Lemma \ref{lemma25}, maximum principle(\cite[Theorem 3.3]{CIL92}), and since there exists $ x_{0} \in B_{1} $ such that $ u(x_{0}) = 0 $, it leads to
\begin{equation*}
  u(x) \leq \sup_{B^{+}} u = \sup_{\partial B^{+}} u = 0,
\end{equation*}
which contradicts the definition of $ B^{+} $.

Analogously, consider an open set $ B^{-} = \{ x \in B_{1}| u(x) < 0  \} $, and assume that $ B^{-} $ is nonempty, then
\begin{equation*}
  |Du|^{p} F(D^{2}u,x) \leq 0 \quad \text{in} \quad B^{-}.
\end{equation*}

As before, by Lemma \ref{lemma25}, maximum principle (\cite[Theorem 3.3]{CIL92}), and since there exists $ x_{0} \in B_{1} $ such that $ u(x_{0}) = 0 $, we derive that
\begin{equation*}
  u(x) \geq \inf_{B^{-}} u = \inf_{\partial B^{-}} u = 0,
\end{equation*}
which also contradicts the definition of $ B^{-} $. Consequently, we finish the proof of Theorem~\ref{Thm7}.
\end{proof}

\vspace{3mm}

\section{Examples and remarks}\label{Section9}

In this section, we give some examples and remarks.
\begin{Example}
  A natural extension of our results arises when considering equations given by
\begin{equation*}\label{Section8:eq1}
\left\{
     \begin{aligned}
     &  |Du|^{p} F(D^{2}u,x) =  f(v) \ \ \text{in} \ \ B_{1}   \\
     &  |Dv|^{q} G(D^{2}v,x) = g(u) \ \ \text{in} \ \ B_{1},        \\
     \end{aligned}
     \right.
\end{equation*}
where $ f,g \in C^{0}(\mathbb{R}^{+}) $ and
\begin{equation}\label{Appendix: eq1}
0 \leq f(\delta t) \leq M_{1} \delta^{\lambda_{1}} f(t) \ \ \text{and} \ \ 0 \leq g(\delta t) \leq M_{2} \delta^{\lambda_{2}} g(t),
\end{equation}
with $ M_{1}, M_{2} >0 $, $ 0 \leq   p, q < \infty $, $ \lambda_{1}\lambda_{2}  < (1+p)(1+q) $, $ t >0 $ bounded, and $ \delta > 0 $ small enough. In addition, we assume that
\begin{equation}\label{Appendix: eq2}
f,g \ \text{are non-decreasing}.
\end{equation}

{\bf $ (\mathrm{I}) $.} The results obtained in Sections \ref{Section 3}--\ref{Section 6} remain true, without changes the proof even when $ f, g $ have some bounded coefficients, namely,
\begin{equation*}
\left\{
     \begin{aligned}
     &  |Du|^{p} F(D^{2}u,x) =  f(x, v) \ \ \text{in} \ \ B_{1}   \\
     &  |Dv|^{q} G(D^{2}v,x) = g(x, u) \ \ \text{in} \ \ B_{1},        \\
     \end{aligned}
     \right.
\end{equation*}
where $ f, g $ still satisfy some properties \eqref{Appendix: eq1}--\eqref{Appendix: eq2}.

\vspace{1mm}

{\bf $ (\mathrm{II}) $.} Note that the case $ f(v) = v_{+}^{\lambda_{1}} $, $ g(u) = u_{+}^{\lambda_{2}}    $ corresponds \eqref{DCP}. In addition, $ f, g $ also contain some interesting models, for example,

$ \bullet $ $ f(v) = e^{v}-1 $ and $ g(u) = e^{u} -1 $ with $ M_{1} = M_{2} = 1 $ and $ \lambda_{1} = \lambda_{2} = 0 $;

$ \bullet $ $ f(v) = \log(1+v^{3}) $ and $ g(u) = \log(1+u^{3})  $ with $ M_{1} = M_{2} = 1 $ and $ \lambda_{1} = \lambda_{2} = 0 $,
see \cite[Section 7]{DT20} for more details.
\end{Example}

\begin{Example}
\label{example0}
A classical model is the following:
\begin{equation*}
  |Du|^{p} F(D^{2}u,x) = \sum_{i=1}^{l} c_{i} |x|^{\alpha_{i}} (u)_{+}^{\mu_{i}}(x) \ \ \text{in} \ \ B_{1},
\end{equation*}
where
\begin{equation*}
  c_{i} \geq 0, \ 0 \leq \mu_{i} < \infty \ \ \text{and} \ \ 0< \alpha_{i} < \infty  \ \ \text{for} \ \ 1 \leq i \leq l.
\end{equation*}
In this case, $ f(x, u) = \sum_{i=1}^{l} c_{i} |x|^{\alpha_{i}} (u)_{+}^{\mu_{i}}(x)$ satisfies \hyperref[Section1:ass5]{\bf (A5)} as follows:
\begin{align*}
  & |f(r|x|, su)| \leq  r^{\min_{1\leq i \leq l}\{\alpha_{i}\}} s^{\min_{1\leq i \leq l}\{\mu_{i}\}} ||f_{0}||_{L^{\infty}(B_{1})}   \\
  & \big[resp. |f(r|x|, su)| \leq  r^{\max_{1\leq i \leq l}\{\alpha_{i}\}} s^{\max_{1\leq i \leq l}\{\mu_{i}\}} ||f_{0}||_{L^{\infty}(B_{1})} \big],
\end{align*}
where
\begin{equation*}
  ||f_{0}||_{L^{\infty}(B_{1})} := \sup_{\overline{B}_{1}} |f(|x|, u(x))|.
\end{equation*}
\end{Example}

\begin{Example}\label{example1}
Consider
\begin{equation*}
  u(x_{1}, x_{2},\cdots, x_{n}) = \bigg[ \frac{(1+p-\mu)^{2+p}}{ [(n-2)(1+p-\mu)+(2+p+\alpha)]    (p+2+\alpha)^{1+p}}  \bigg] |x_{i}|^{\frac{2+p+\alpha}{1+p-\mu}}, \ \ x \in B_{1}.
\end{equation*}
Simple computations yield that $ u $ satisfies
\begin{equation*}
|Du|^{p} \Delta u =   \left\{
     \begin{aligned}
     &  |x_{i}|^{\alpha}(u)_{+}^{\mu}  \lesssim |x|^{\alpha}(u)_{+}^{\mu}    \ \  \text{in}   \ \  B_{1},  \ \ \text{if}   \ \   \alpha, \mu \neq 0,          \\
     &  |x_{i}|^{\alpha} \lesssim |x|^{\alpha}  \qquad \qquad \ \  \text{in} \ \ B_{1}, \ \ \text{if} \ \ \alpha \neq 0,  \mu = 0,     \\
     & (u)_{+}^{\mu} \qquad \qquad  \qquad  \ \ \ \   \text{in} \ \ B_{1}, \ \ \text{if} \ \  \alpha = 0, \mu \neq 0,   \\
     & 1 \qquad \qquad \qquad \qquad \ \ \   \text{in} \ \ B_{1}, \ \ \text{if} \ \  \alpha, \mu = 0.
     \end{aligned}
     \right.
\end{equation*}

In particular, we have that $ u \in C^{1,\frac{1+\alpha+\mu}{1+p-\mu}}_{\mathrm{loc}}(B_ {1}) $. Furthermore, if the condition $ 0 < \alpha < \infty $ in \hyperref[Section1:ass5]{\bf (A5)} also holds, then we have $ 1+ \alpha + \mu > 0 $.
\end{Example}

\begin{remark}\label{rk91}
It is essential to highlight some consequences arising from Example \ref{example1}. Let $ \tau = \frac{1+\alpha+\mu}{1+p-\mu} $, then we note that at critical point $ 0 \in \mathcal{S}_{u}(B_{1})$, the following relationships hold:
\begin{equation*}
 \left\{
     \begin{aligned}
     & \tau \in (0,1) \ \ \Leftrightarrow \ \ \alpha < p-2\mu \ \ \Rightarrow \ \ u \ \text{belongs to} \ C^{1,\tau} \ \text{at the origin},                \\
     & \tau =1  \qquad \ \Leftrightarrow \ \ \alpha = p-2\mu \ \ \Rightarrow \ \ u \ \text{exhibits a quadratic decay at the origin},                 \\
     & \tau >1 \qquad  \  \Leftrightarrow \ \  \alpha >  p-2\mu \ \ \Rightarrow \ \ u \ \text{exhibits a super-quadratic decay at the origin}.
     \end{aligned}
     \right.
\end{equation*}

In particular, viscosity solutions enjoy classical estimates along the set of critical points $ \mathcal{S}_{u}(B_{1}) $, provided $ \alpha > p-2\mu $.

\end{remark}

\vspace{2mm}

\begin{Example}\label{example2}
A general model is formulated as follows:
\begin{equation}\label{Appendix: eq3}
  |Du|^{p} F(D^{2}u, x) = \sum_{i=1}^{N} \mathrm{dist}^{\alpha_{i}}(x, G_{i})(u)_{+}^{\mu_{i}}  +  g_{i}(|x|) \ \ \text{in} \ \ B_{1},
\end{equation}
where $ G_{i} \subset B_{1} $ are disjoint closed sets, $ \mu_{i} \in [0, 1+p)$, $ 0 < \alpha_{i} < \infty  $ and
\begin{equation*}
  \limsup_{|x|\rightarrow 0} \frac{g_{i}(|x|)}{\mathrm{dist}^{l_{i}}(x, G_{i})} = L_{i} \in [0, \infty) \ \ \text{for some} \ \ l_{i} \geq 0.
\end{equation*}
Then viscosity solution to problem \eqref{Appendix: eq3} belongs to $ C_{\mathrm{loc}}^{\min_{1\leq i \leq N}\big\{\frac{2+p+\alpha_{i}}{1+p-\mu_{i}}, \frac{2+p+l_{i}}{1+p}  \big\}}$ along the sets $ \bigcap_{i=1}^{N}G_{i} \cap \mathcal{S}_{u}(B_{1}) $, where $ \mathcal{S}_{u}(B_{1}) = \{ x \in B_{1}: u(x)=|Du| = 0\} $.
\end{Example}

\begin{Example}[{\bf Radial Scenario}]
\label{Intro:ex3}
Consider the function $ \Phi: B_{1} \rightarrow \mathbb{R} $ given by
\begin{equation*}
  \Phi(x) = C_{1}(n, p, \mu, \alpha, \Lambda) |x|^{{\frac{2+p+\alpha}{1+p-\mu}}},
\end{equation*}
where $ C_{1}(n, p, \mu, \alpha, \Lambda) $ is a positive constant to be determined.

Direct calculations yield that
\begin{equation*}
  D\Phi(x) = \frac{2+p+\alpha}{1+p-\mu} C_{1} |x|^{\frac{1+\alpha+\mu}{1+p-\mu}}\frac{x}{|x|},
\end{equation*}
and
\begin{equation*}
  D^{2} \Phi(x) = C_{1}\frac{2+p+\alpha}{1+p-\mu} \bigg[ \frac{1+\alpha+\mu}{1+p-\mu} |x|^{\frac{2\mu+\alpha-p}{1+p-\mu}}\frac{x\otimes x}{|x|^{2}}   + |x|^{\frac{2\mu+\alpha-p}{1+p-\mu}} \big(\textbf{I}d_{n}-\frac{x\otimes x}{|x|^{2}}\big)   \bigg].
\end{equation*}
By \eqref{1a} and \eqref{3a}, we have that
\begin{equation*}
  F(D^{2}\Phi, x) \leq C_{1} \Lambda \frac{(p+2+\alpha)\big[n(1+p-\mu)+ (2\mu+\alpha-p)  \big]}{(1+p-\mu)^{2}} |x|^{\frac{\alpha+2\mu-p}{1+p-\mu}}.
\end{equation*}
Thus, if we choose
\begin{equation*}
  C_{1}(n, p, \mu, \alpha, \Lambda):= \bigg[\frac{(1+p-\mu)^{2+p}}{\Lambda(2+p+\alpha)^{1+p}[n(1+p-\mu)+(2\mu+\alpha-p)]}            \bigg]^{\frac{1}{1+p-\mu}},
\end{equation*}
it follows that $ \Phi $ is a viscosity super-solution to
\begin{equation*}
  |D\Phi|^{p} F(D^{2}\Phi, x) = |x|^{\alpha}(\Phi)_{+}^{\mu}  \ \ \text{in} \ \ B_{1}.
\end{equation*}
Particularly, it deduces that $ \Phi $ belongs to $ C^{1, \frac{1+\alpha+\mu}{1+p-\mu}}_{\mathrm{loc}}(B_{1})  $.
\end{Example}

\begin{remark}
The method of Theorem \ref{Thm5} also be applicable to
\begin{equation}\label{Appendix: eq5}
|Du|^{\gamma} \Delta_{p}^{\rm{N}} u = |x|^{\alpha} (u)_{+}^{\mu} \ \ \text{in} \ \ B_{1},
\end{equation}
where $ -1 < \gamma < \infty $, $ 1 < p < \infty $, $ 0 \leq \mu < 1+\gamma $ and $ 0 < \alpha < \infty $. The viscosity solutions to problem \eqref{Appendix: eq5} should belong to the class $ C^{\frac{2+\gamma+\alpha}{1+\gamma-\mu}}_{\mathrm{loc}} $ along the critical set $ \mathcal{S}(u) $, see \cite[Theorem 6.1]{AdaS25} for further details.
\end{remark}

\begin{remark}
It is natural to consider a more general form of the degeneracy law:
\begin{equation*}
  \mathcal{H}_{i}: B_{1} \times \mathbb{R}^{n} \rightarrow \mathbb{R}, \ \ i=1,2,
\end{equation*}
satisfying for some constants $ c_{0}, c_{1} > 0 $,
\begin{equation*}
  c_{0}|\xi|^{p_{i}} \leq \mathcal{H}_{i}(x,\xi) \leq c_{1} |\xi|^{p_{i}}, \ \ \forall (x,\xi) \in B_{1} \times \mathbb{R}^{n}.
\end{equation*}
\end{remark}

\appendix

\section{Proof of Theorem \ref{Se2:Thm2.2}}
\label{Appendix A}

In this appendix, we are devoted to presenting the proof of Theorem \ref{Se2:Thm2.2}.   

\begin{proof}[{\bf Proof of Theorem \ref{Se2:Thm2.2}}]
The idea of the proof is inspired by \cite[Theorem 5.3]{BM12} and references therein. We begin with condition $ (\mathrm{i}) $ to derive an upper bound. Choose an $ \epsilon >0 $ (to be determined later). From condition $ (\mathrm{i}) $, there exists $ t_{\alpha} > 0 $, depending on $ f, p $ and $ \epsilon $, such that
\begin{equation}\label{Appendix:eq1}
  f(x,t) \geq - \epsilon t^{1+p},  \quad   (x,t) \in B_{1} \times [t_{\alpha}, \infty).
\end{equation}

For a given viscosity solution $ u \in C(\overline{B}_{1}) $ to \eqref{Se2:D-BVP}, we consider the following open subset of $ B_{1} $
\begin{equation*}
 \widehat{B}_{\alpha}:= \{x \in B_{1}: u(x) > t_{\alpha}   \}.
\end{equation*}

Now we proceed by distinguishing two cases.

\textbf{Case 1.} If $ \widehat{B}_{\alpha} = \emptyset $, then the conclusion holds in this case.

\textbf{Case 2.} If $ \widehat{B}_{\alpha} \neq \emptyset $, without loss of generality, we assume $ t_{\alpha} \geq \sup_{\partial B_{1}} g $. Note that $ t_{\alpha} < M:= \sup_{B_{1}} u $. Let $ v \in C(\overline{B}_{1}) $ be a viscosity solution of the following Dirichlet problem:
\begin{equation}
\label{Appendix:eq2}
\left\{
     \begin{aligned}
     & |Dv|^{p} F(D^{2}v,x) = - \epsilon M^{1+p}   \quad \mathrm{in} \ \ B_{1}          \\
     &  v = t_{\alpha}   \quad \quad \quad \quad \quad \quad \quad \quad \quad \quad  \mathrm{on}   \ \ \partial B_{1}.        \\
     \end{aligned}
     \right.
\end{equation}
By the comparison principle (Lemma \ref{Sec2:lemma4}), we have $ v \geq t_{\alpha} $ in $ B_{1} $. Moreover, applying the ABP estimate (\cite[Theorem 1]{GPA09}) yields
\begin{equation}
\label{Appendix:eq3}
\sup_{B_{1}} v \leq C(t_{\alpha}+ \epsilon^{\frac{1}{1+p}}M),
\end{equation}
where $ C > 0 $ depends only on $ n, \lambda, \Lambda $ and $ p $. Next, consider the following Dirichlet problem in subdomain $ \widehat{B}_{\alpha} $:
 \begin{equation*}
\left\{
     \begin{aligned}
     & |Du|^{p} F(D^{2}u,x) = f(x,u)   \quad \mathrm{in} \ \ \widehat{B}_{\alpha}          \\
     &  u = t_{\alpha}   \quad \quad \quad \quad \quad \quad \quad \quad \quad   \mathrm{on}   \ \ \partial \widehat{B}_{\alpha},        \\
     \end{aligned}
     \right.
\end{equation*}
where $ f(x,u) \geq - \epsilon u^{1+p} \geq - \epsilon M^{1+p} $. Applying the comparison principle (Lemma \ref{Sec2:lemma4}) once more gives $ u \leq v $ in $ \widehat{B}_{\alpha} $, and hence
\begin{equation}
\label{Appendix:eq4}
  M \leq \sup_{B_{1}} v.
\end{equation}

Now we combine \eqref{Appendix:eq3} and \eqref{Appendix:eq4} to obtain
\begin{equation}\label{Appendix:eq5}
  M \leq \sup_{B_{1}} v \leq C(t_{\alpha}+ \epsilon^{\frac{1}{1+p}}M).
\end{equation}
By choosing $ \epsilon = (\frac{2t_{\alpha}}{M})^{1+p} $, then \eqref{Appendix:eq5} leads to
\begin{equation*}
  u \leq M \leq 3Ct_{\alpha}.
\end{equation*}
Hence the desired estimate follows in this case. For condition $ (\mathrm{ii}) $, a similar argument gives a corresponding lower bound $ u \geq 3Ct_{\alpha} $ in $ B_{1} $. This completes the proof of Theorem \ref{Se2:Thm2.2}.
\end{proof}

\noindent {\bf Acknowledgments} The authors express their gratitude to Prof. Eduardo V. Teixeira for his interest in this paper. They also sincerely thank the anonymous referees for their valuable and insightful comments, which have greatly improved the quality of the manuscript.

\vspace{2mm}

\noindent {\bf Author Contributions} We declare that all authors have reviewed and contributed equally to this manuscript.

\vspace{2mm}

\noindent {\bf Funding} This work was supported by the National Natural Science Foundation of China (No. 12271093) and the Jiangsu Provincial Scientific Research Center of Applied Mathematics (Grant No. BK20233002). The second author is funded by Shanghai Institute for Mathematics and Interdisciplinary Sciences (SIMIS) under grant number SIMIS-ID-2025-AD.

\vspace{2mm}

\noindent {\bf Data Availability} No datasets were generated or analyzed during the current study.

\vspace{2mm}

\noindent {\bf \Large Declarations}

\vspace{2mm}

\noindent {\bf Competing Interests} The authors declare no competing interests.

\vspace{2mm}

\noindent {\bf Ethical Approval} Not applicable

\end{sloppypar}
\end{document}